\documentclass[leqno,11pt]{article}%
\usepackage{amsfonts}
\usepackage{amsmath}
\usepackage{amssymb}
\usepackage{graphicx}
\usepackage{palatino}
\usepackage[colorlinks]{hyperref}%
\providecommand{\U}[1]{\protect\rule{.1in}{.1in}}
\newtheorem{theorem}{Theorem}

\newtheorem{corollary}[theorem]{Corollary}

\newtheorem{definition}[theorem]{Definition}

\newtheorem{lemma}[theorem]{Lemma}

\newtheorem{proposition}[theorem]{Proposition}
\newtheorem{remark}[theorem]{Remark}

\newenvironment{proof}[1][Proof]{\noindent\textbf{#1.} }{\ \rule{0.5em}{0.5em}}
\begin{document}

\title{{\huge Multivalued backward stochastic differential equations with oblique
subgradients}\thanks{The work for this paper was supported by founds from the
Grant PN-II-ID-PCE-2011-3-0843, \textit{Deterministic and stochastic systems
with state constraints}, no. 241/05.10.2011.\newline$\clubsuit$Corresponding
author\newline e-mail: masigassous@gmail.com, aurel.rascanu@uaic.ro,
eduard.rotenstein@uaic.ro}\medskip}
\author{Anouar Gassous$^{\text{ }a}$, Aurel R\u{a}\c{s}canu$^{\text{ }a,b}$, Eduard
Rotenstein$^{\text{ }a,\clubsuit}$\bigskip\\$^{a\text{ }}${\small Faculty of Mathematics, "Alexandru Ioan Cuza"
University, Bd. Carol I no. 9-11, Ia\c{s}i, Rom\^{a}nia}\\$^{b\text{ }}${\small Institut of Mathematics of the Romanian Academy,
Ia\c{s}i branch, Bd. Carol I no. 8, Rom\^{a}nia}\medskip}
\maketitle

\begin{abstract}
We study the existence and uniqueness of the solution for the following
backward stochastic variational inequality with oblique reflection (for short,
$BSVI\left(  H(t,y),\varphi,F\right)  $), written under differential form

\[
\left\{
\begin{array}
[c]{l}%
-dY_{t}+H\left(  t,Y_{t}\right)  \partial\varphi\left(  Y_{t}\right)  \left(
dt\right)  \ni F\left(  t,Y_{t},Z_{t}\right)  dt-Z_{t}dB_{t},\quad t\in\left[
0,T\right]  ,\smallskip\\
Y_{T}=\eta,
\end{array}
\right.
\]
where $H$ is a bounded symmetric smooth matrix and $\varphi$ is a proper
convex lower semicontinuous function, with $\partial\varphi$ being its
subdifferential operator. The presence of the product $H\partial\varphi$ does
not permit the use of standard techniques because it does conserve neither the
Lipschitz property of the matrix nor the monotonicity property of the
subdifferential operator. We prove that, if we consider the dependence of $H$
only on the time, the equation admits a unique strong solution and, allowing
the dependence also on the state of the system, the above $BSVI\left(
H(t,y),\varphi,F\right)  $ admits a weak solution in the sense of the
Meyer-Zheng topology. However, for that purpose we must renounce at the
dependence on $Z$ for the generator function and we situate our problem in a
Markovian framework.\bigskip

\end{abstract}

\bigskip

\textbf{2010} \textbf{AMS Subject Classification: }60H10, 60H30, 93E03\bigskip

\section{Introduction}

Backward stochastic differential equations (in short BSDE's) were first
introduced by Bismut in 1973 in the paper \cite{Bismut:73} as equation for the
adjoint process in the stochastic version of Pontryagin maximum principle. In
1990, Pardoux and Peng \cite{Pardoux/Peng:90} generalized and consecrated the
well known now notion of nonlinear backward stochastic differential equation
and they provided existence and uniqueness results for the solution of this
kind of equation. Starting with the paper of Pardoux and Peng
\cite{Pardoux/Peng:92}, a stochastic approach to the existence problem of a
solution for many types of deterministic partial differential equations has
been developed. Since then the interest in BSDEs has kept growing, both in the
direction of generalization of the emerging equations and construction of
approximation schemes for them. BSDEs have been widely used as a very useful
instrument for modelling various physical phenomena, in stochastic control and
especially in mathematical finance, as any pricing problem, by replication,
can be written in terms of linear BSDEs, or non-linear BSDEs with portfolios
constraints. Pardoux and R\u{a}\c{s}canu \cite{Pardoux/Rascanu:98} proved,
using a probabilistic interpretation, the existence of the viscosity solution
for a multivalued PDE (with subdifferential operator) of parabolic and
elliptic type.

Backward stochastic variational inequalities (for short, BSVIs) were first
analyzed by Pardoux and R\u{a}\c{s}canu in \cite{Pardoux/Rascanu:98} and
\cite{Pardoux/Rascanu:99} (the extension for Hilbert spaces case), by using a
method that consisted of a penalizing scheme, followed by its convergence.
Even though this type of penalization approach is very useful when dealing
with multivalued backward stochastic dynamical systems governed by a
subdifferential operator, it fails when dealing with a general maximal
monotone operator. This motivated a new approach for the later case of
equations, via convex analysis instruments. In \cite{Rascanu/Rotenstein:11},
R\u{a}\c{s}canu and Rotenstein established, using the Fitzpatrick function, a
one-to-one correspondence between the solutions of those types of equations
and the minimum points of some proper, convex, lower semicontinuous functions,
defined on well-chosen Banach spaces.

Multi-dimensional backward stochastic differential equations with oblique
reflection (in fact BSDEs reflected on the boundary of a special unbounded
convex domain along an oblique direction), which arises naturally in the study
of optimal switching problem were recently studied by Hu and Tang in
\cite{Hu/Tang:10}. As applications, the authors apply the results to solve the
optimal switching problem for stochastic differential equations of functional
type, and they give also a probabilistic interpretation of the viscosity
solution to a system of variational inequalities.

It worth mentioning that, until now, even for quite complex problems like the
ones analyzed by Maticiuc and R\u{a}\c{s}canu in \cite{Maticiuc/Rascanu:07} or
\cite{Maticiuc/Rascanu:10}, when dealing with BSVIs, the reflection was made
upon the normal direction at the frontier of the domain and it was caused by
the presence of the subdifferential operator of a convex lower semicontinuous
function. As the main achievement of this paper we prove the existence and
uniqueness of the solution for the more general BSVI with oblique subgradients%
\[
\left\{
\begin{array}
[c]{l}%
-dY_{t}+H\left(  t,Y_{t}\right)  \partial\varphi\left(  Y_{t}\right)  \left(
dt\right)  \ni F\left(  t,Y_{t},Z_{t}\right)  dt-Z_{t}dB_{t},\quad t\in\left[
0,T\right]  ,\smallskip\\
Y_{T}=\eta,
\end{array}
\right.
\]
where $B$ is a standard Brownian motion defined on a complete probability
space, $F$ is the generator function and the random variable $\eta$ is the
terminal data. The term $H(X)$ acts on the set of subgradients, fact which
will determine a oblique reflection for the feedback process. A similar setup
was constructed and studied for forward stochastic variational inequalities by
Gassous, R\u{a}\c{s}canu and Rotenstein in
\cite{Gassous/Rascanu/Rotenstein:12} by considering first a (deterministic)
generalized Skorokhod problem with oblique subgradients, prior to the general
stochastic case. In the current paper the problems also rise when we operate
with the product $H\left(  t,Y_{t}\right)  \partial\varphi\left(
Y_{t}\right)  $, which does not inherit neither the monotonicity of the
subdifferential operator nor the Lipschitz property of the matrix involved,
problems which will be overcome by using different methods compared to the
ones used for subgradients reflected upon the normal direction. We will split
our problem into two new ones. For the situation when we have only a time
dependence for the matrix $H$ we obtain the existence of a strong solution,
together with the existence of a absolutely continuous feedback-subgradient
process. However, for the general case of a state dependence for $H$ we will
use tightness criteria in order to get a solution for the equation. In
\cite{Buckdahn/Engelbert/Rascanu:04}, Buckdahn, Engelbert and R\u{a}\c{s}canu
discussed the concept of weak solutions of a certain type of backward
stochastic differential equations (not multivalued). Using weak convergence in
the Meyer--Zheng topology, they provided a general existence result. We will
put also our problem into a Markovian framework. The problem consists in
answering in which sense can we take the limit in the sequence $\{\left(
Y^{n},Z^{n},U^{n}\right)  \}_{n}$, given by the solutions of the approximating
equations. We have to prove that it is tight in a certain topology. Even the
$S-$topology introduced by Jakubowski in \cite{Jakubowski:97} (and used for
similar setups by Boufoussi and Casteren \cite{Boufoussi/Casteren:04} or LeJai
\cite{LeJai:02}) seems suitable for our context, the regularity of the
subgradient process given by the approximating equation as part of its
solution permits us to show a convergence in the sense of the Meyer-Zheng
topology, that is the laws converge weakly if we equip the space of paths with
the topology of convergence in $dt-$measure. The tightness of $\{Z^{n}\}_{n}$
is hard to get, therefore we renounce at the dependence on $Z$ for the
generator function of the multivalued backward equation. This framework
permits also to analyze the existence of viscosity solutions for systems of
parabolic variational inequalities driven by generalized subgradients.

The article is organized as follows. Section 2 presents the framework of our
study, the assumptions and the hypotheses on the data, the notions of weak and
strong solution for the equations and it closes with the enunciations of the
main results of the paper, the complete proofs representing the core of
Sections 4 and 5. Section 3 is dedicated to some useful apriori estimates for
the solutions of the approximating equations. Section 4 proves the strong
existence and uniqueness of the solution when the matrix $H$ does not depend
on the state of the system, while Section 5 deals with the existence of a weak
solution for the general case of $H=H(t,y)$. For the clarity of the
presentation, the last part of the paper groups together, under the form of an
Annex with three subsections, some useful results that are used throughout
this article.

\section{Setting the problem}

This section is dedicated to the construction of the problem that we will
study in the sequel. We present the hypothesis imposed on the coefficients and
we formulate the main results of this article. The proofs will be detailed in
the next three sections.

Let $T>0$ be fixed and consider the backward stochastic variational inequality
with oblique reflection (for short, we will write $BSVI\left(  H\left(
t,y\right)  ,\varphi,F\right)  $, $BSVI\left(  H\left(  t\right)
,\varphi,F\right)  $ or, respectively, $BSVI\left(  H\left(  y\right)
,\varphi,F\right)  $ if the matrix $H$ depends only on time or, respectively,
on the state of the system), $\mathbb{P}-a.s.~\omega\in\Omega$,%
\begin{equation}
\left\{
\begin{array}
[c]{l}%
Y_{t}+%
%TCIMACRO{\dint _{t}^{T}}%
%BeginExpansion
{\displaystyle\int_{t}^{T}}
%EndExpansion
H\left(  s,Y_{s}\right)  dK_{s}=\eta+%
%TCIMACRO{\dint _{t}^{T}}%
%BeginExpansion
{\displaystyle\int_{t}^{T}}
%EndExpansion
F\left(  s,Y_{s},Z_{s}\right)  ds-%
%TCIMACRO{\dint _{t}^{T}}%
%BeginExpansion
{\displaystyle\int_{t}^{T}}
%EndExpansion
Z_{s}dB_{s},\quad t\in\left[  0,T\right]  ,\smallskip\\
dK_{s}\in\partial\varphi\left(  Y_{s}\right)  \left(  ds\right)  ,
\end{array}
\right.  \label{main oblique BSVI}%
\end{equation}
where

\begin{enumerate}
\item[$\left(  H_{1}\right)  $] $\quad\left(  \Omega,\mathcal{F}%
,\mathbb{P},\{\mathcal{F}_{t}\}_{t\geq0}\right)  $ is a stochastic basis and
$\{B_{t}:t\geq0\}$ is a $\mathbb{R}^{k}-$valued Brownian motion. Moreover,
$\mathcal{F}_{t}=\mathcal{F}_{t}^{B}=\sigma(\{B_{s}:0\leq s\leq t\})\vee
\mathcal{N}$.

\item[$\left(  H_{2}\right)  $] $\quad H(\cdot,\cdot,y):\Omega\times
\mathbb{R}_{+}\rightarrow\mathbb{R}^{d\times d}$ is progressively measurable
for every $y\in\mathbb{R}^{d}$; there exists $\Lambda,b>0$ such that
$\mathbb{P-}a.s.$ $\omega\in\Omega$, $H=\left(  h_{i,j}\right)  _{d\times
d}\in C^{1,2}\left(  \mathbb{R}_{+}\mathbb{\times R}^{d};\mathbb{R}^{d\times
d}\right)  $ and, for all $t\in\left[  0,T\right]  $ and $y,\tilde{y}%
\in\mathbb{R}^{d}$, $\mathbb{P-}a.s.$ $\omega\in\Omega$,%
\begin{equation}
\left\{
\begin{array}
[c]{ll}%
\left(  i\right)  \quad & h_{i,j}\left(  t,y\right)  =h_{j,i}\left(
t,y\right)  ,\quad\forall i,j\in\overline{1,d},\medskip\\
\left(  ii\right)  \quad & \left\langle H\left(  t,y\right)  u,u\right\rangle
\geq a\left\vert u\right\vert ^{2},\quad\forall u\in\mathbb{R}^{d}\text{ (for
some }a\geq1\text{)},\medskip\\
\left(  iii\right)  \quad & |H\left(  t,\tilde{y}\right)  -H\left(
t,y\right)  |+|\left[  H\left(  t,\tilde{y}\right)  \right]  ^{-1}-\left[
H\left(  t,y\right)  \right]  ^{-1}|\leq\Lambda|\tilde{y}-y|,\medskip\\
\left(  iv\right)  \quad & |H\left(  t,y\right)  |+|\left[  H\left(
t,y\right)  \right]  ^{-1}|~\leq b,
\end{array}
\right.  \label{hypothesis on H}%
\end{equation}
where $\left\vert H\left(  x\right)  \right\vert \overset{def}{=}\left(
\sum_{i,j=1}^{d}\left\vert h_{i,j}\left(  x\right)  \right\vert ^{2}\right)
^{1/2}$. We denoted by $\left[  H\left(  t,y\right)  \right]  ^{-1}$ the
inverse matrix of $H\left(  t,y\right)  $. Therefore, $\left[  H\left(
t,y\right)  \right]  ^{-1}$ has the same properties (\ref{hypothesis on H}%
$-\left(  i\right)  ,\left(  ii\right)  $) as $H\left(  t,y\right)  $.

\item[$\left(  H_{3}\right)  $] $\quad$the function%
\[
\varphi:\mathbb{R}^{d}\rightarrow\left]  -\infty,+\infty\right]  \text{ is a
proper lower semicontinuous convex function.}%
\]

\end{enumerate}

The generator function $F\left(  \cdot,\cdot,y,z\right)  :\Omega\times\left[
0,T\right]  \rightarrow\mathbb{R}^{d}$ is progressively measurable for every
$\left(  y,z\right)  \in\mathbb{R}^{d}\times\mathbb{R}^{d\times k}$ and there
exist $L,\ell,\rho\in L^{2}\left(  0,T;\mathbb{R}_{+}\right)  $ such that%
\begin{equation}
\left\{
\begin{array}
[c]{l}%
\begin{array}
[c]{l}%
\left(  i\right)  \quad\text{\textit{Lipschitz conditions: for all}
}y,y^{\prime}\in\mathbb{R}^{d},\;z,z^{\prime}\in\mathbb{R}^{d\times
k},\;d\mathbb{P}\otimes dt-a.e.:\medskip\\
\quad\quad\quad%
\begin{array}
[c]{l}%
\left\vert F(t,y^{\prime},z)-F(t,y,z)\right\vert \leq L\left(  t\right)
|y^{\prime}-y|\text{,}\medskip\\
|F(t,y,z^{\prime})-F(t,y,z)|\leq\ell\left(  t\right)  |z^{\prime}-z|\text{;}%
\end{array}
\end{array}
\medskip\\%
\begin{array}
[c]{l}%
\left(  ii\right)  \quad\text{\textit{Boundedness condition:}}\medskip\\
\quad\quad~%
\begin{array}
[c]{ll}
& \left\vert F\left(  t,0,0\right)  \right\vert \leq\rho\left(  t\right)
,\quad d\mathbb{P}\otimes dt-a.e.\text{.}%
\end{array}
\end{array}
\end{array}
\right.  \tag{$H_4$}\label{hypothesis on F}%
\end{equation}
$\smallskip$

Denote by $\partial\varphi$ the subdifferential operator of $\varphi$:%
\[
\partial\varphi\left(  x\right)  \overset{def}{=}\left\{  \hat{x}\in
\mathbb{R}^{d}:\left\langle \hat{x},y-x\right\rangle +\varphi\left(  x\right)
\leq\varphi\left(  y\right)  ,\text{ for all }y\in\mathbb{R}^{d}\right\}
\]
and $Dom\left(  \partial\varphi\right)  =\{x\in\mathbb{R}^{d}:\partial
\varphi\left(  x\right)  \neq\emptyset\}$. We will use the notation
$(x,\hat{x})\in\partial\varphi$ in order to express that $x\in Dom\left(
\partial\varphi\right)  $ and $\hat{x}\in\partial\varphi\left(  x\right)  $.
The vector given by the quantity $H\left(  x\right)  \hat{x}$, with $\hat
{x}\in\partial\varphi\left(  x\right)  $ will be called in what follows
\textit{oblique subgradient}.

\begin{remark}
If $E$ is a closed convex subset of $\mathbb{R}^{d}$ then the convex indicator
function%
\[
\varphi\left(  x\right)  =I_{E}\left(  x\right)  =\left\{
\begin{array}
[c]{cc}%
0, & \text{if }x\in E,\smallskip\\
+\infty, & \text{if }x\notin E,
\end{array}
\right.
\]
is a convex l.s.c. function and, for $x\in E$,%
\[
\partial I_{E}\left(  x\right)  =\{\hat{x}\in\mathbb{R}^{d}:\left\langle
\hat{x},y-x\right\rangle \leq0,\quad\forall y\in E\}=N_{E}\left(  x\right)  ,
\]
where $N_{E}\left(  x\right)  $ is the closed external normal cone to $E$ at
$x$. We have $N_{E}\left(  x\right)  =\emptyset$ if $x\notin E$ and
$N_{E}\left(  x\right)  =\{0\}$ if $x\in int(E)$ (we denote by $int(E)$ the
interior of the set $E$).
\end{remark}

\noindent We shall call \textit{oblique reflection directions} at time $t$ the
vectors given by%
\[
\nu_{t,x}=H\left(  t,x\right)  n_{x},\quad x\in Bd\left(  E\right)  ,
\]
where $n_{x}\in N_{E}\left(  x\right)  $ (we denote by $Bd(E)$ the boundary of
the set $E$).

Let $k:\left[  t,T\right]  \rightarrow\mathbb{R}^{d}$, where $0\leq t\leq T$.
We denote, $\left\Vert k\right\Vert _{\left[  t,T\right]  }\overset{def}%
{=}\sup\left\{  \left\vert k\left(  s\right)  \right\vert :t\leq s\leq
T\right\}  $ and, for $t=0$, $\left\Vert k\right\Vert _{T}\overset{def}%
{=}\left\Vert k\right\Vert _{\left[  0,T\right]  }$. Considering
$\mathcal{D}\left[  t,T\right]  $ the set of the partitions of the time
interval $\left[  t,T\right]  $, of the form $\Delta=(t=t_{0}<t_{1}%
<...<t_{n}=T)$, let%
\[
S_{\Delta}(k)=%
%TCIMACRO{\dsum \limits_{i=0}^{n-1}}%
%BeginExpansion
{\displaystyle\sum\limits_{i=0}^{n-1}}
%EndExpansion
|k(t_{i+1})-k(t_{i})|
\]
and $\left\updownarrow k\right\updownarrow _{\left[  t,T\right]  }%
\overset{def}{=}\sup\limits_{\Delta\in\mathcal{D}}S_{\Delta}(k)$; if $t=0$,
denote $\left\updownarrow k\right\updownarrow _{T}\overset{def}{=}%
\left\updownarrow k\right\updownarrow _{\left[  0,T\right]  }$. We consider
the space of bounded variation functions $BV(\left[  0,T\right]
;\mathbb{R}^{d})=\{k~|~k:\left[  0,T\right]  \rightarrow\mathbb{R}^{d},$
$\left\updownarrow k\right\updownarrow _{T}<\infty\}.$ Taking on the space of
continuous functions $C\left(  \left[  0,T\right]  ;\mathbb{R}^{d}\right)  $
the usual supremum norm, we have the duality connection $(C(\left[
0,T\right]  ;\mathbb{R}^{d}))^{\ast}=\{k\in BV(\left[  0,T\right]
;\mathbb{R}^{d})~|~k(0)=0\}$, with the duality between these spaces given by
the Riemann-Stieltjes integral $\left(  y,k\right)  \mapsto\int_{0}%
^{T}\left\langle y\left(  t\right)  ,dk\left(  t\right)  \right\rangle .$ We
will say that a function $k\in BV_{loc}(\mathbb{R}_{+};\mathbb{R}^{d})$ if,
for every $T>0$, $k\in BV(\left[  0,T\right]  ;\mathbb{R}^{d})$.

\begin{definition}
Given two functions $x,k:\mathbb{R}_{+}\rightarrow$ $\mathbb{R}^{d}$ we say
that $dk\left(  t\right)  \in\partial\varphi\left(  x\left(  t\right)
\right)  \left(  dt\right)  $ if%
\[%
\begin{array}
[c]{ll}%
\left(  a\right)  \quad & x:\mathbb{R}_{+}\rightarrow\mathbb{R}^{d}%
\quad\text{is continuous,}\smallskip\\
\left(  b\right)  \quad &
%TCIMACRO{\dint _{0}^{T}}%
%BeginExpansion
{\displaystyle\int_{0}^{T}}
%EndExpansion
\varphi\left(  x\left(  t\right)  \right)  dt<\infty,\text{ for all }%
T\geq0,\smallskip\\
\left(  c\right)  \quad & k\in BV_{loc}\left(  \mathbb{R}_{+};\mathbb{R}%
^{d}\right)  ,\text{\quad}k\left(  0\right)  =0\text{,}\smallskip\\
\left(  d\right)  \quad &
%TCIMACRO{\dint _{s}^{t}}%
%BeginExpansion
{\displaystyle\int_{s}^{t}}
%EndExpansion
\left\langle y\left(  r\right)  -x(r),dk\left(  r\right)  \right\rangle +%
%TCIMACRO{\dint _{s}^{t}}%
%BeginExpansion
{\displaystyle\int_{s}^{t}}
%EndExpansion
\varphi\left(  x\left(  r\right)  \right)  dr\leq%
%TCIMACRO{\dint _{s}^{t}}%
%BeginExpansion
{\displaystyle\int_{s}^{t}}
%EndExpansion
\varphi\left(  y\left(  r\right)  \right)  dr,\smallskip\\
\multicolumn{1}{r}{} & \multicolumn{1}{r}{\text{for all }0\leq s\leq t\leq
T\quad\text{and }y\in C\left(  \left[  0,T\right]  ;\mathbb{R}^{d}\right)
.\smallskip}%
\end{array}
\]

\end{definition}

We introduce now the notion of solution for Eq.(\ref{main oblique BSVI}). We
will study two types of solution, given by the following Definitions. For the
case $H\left(  t,y\right)  \equiv H\left(  t\right)  $ we obtain the existence
of a strong solution while, for $H\left(  t,y\right)  $ we obtain a weak
solution for Eq.(\ref{main oblique BSVI}).

\begin{definition}
\label{Def of strong solution}Given $\left(  \Omega,\mathcal{F},\mathbb{P}%
,\{\mathcal{F}_{t}\}_{t\geq0}\right)  $ a fixed stochastic basis and
$\{B_{t}:t\geq0\}$ a $\mathbb{R}^{k}-$valued Brownian motion, we state that a
triplet $\left(  Y,Z,K\right)  $ is a strong solution of the $BSVI\left(
H\left(  t\right)  ,\varphi,F\right)  $ if $(Y,Z,K):\Omega\times\left[
0,T\right]  \rightarrow\mathbb{R}^{d}\times\mathbb{R}^{d\times k}%
\times\mathbb{R}^{d}$ are progressively measurable continuous stochastic
processes and $\mathbb{P}-a.s.~\omega\in\Omega$,%
\[
\left\{
\begin{array}
[c]{l}%
Y_{t}+%
%TCIMACRO{\dint _{t}^{T}}%
%BeginExpansion
{\displaystyle\int_{t}^{T}}
%EndExpansion
H\left(  s\right)  dK_{s}=\eta+%
%TCIMACRO{\dint _{t}^{T}}%
%BeginExpansion
{\displaystyle\int_{t}^{T}}
%EndExpansion
F\left(  s,Y_{s},Z_{s}\right)  ds-%
%TCIMACRO{\dint _{t}^{T}}%
%BeginExpansion
{\displaystyle\int_{t}^{T}}
%EndExpansion
Z_{s}dB_{s},\quad\forall t\in\left[  0,T\right]  ,\smallskip\\
dK_{s}\in\partial\varphi\left(  Y_{s}\right)  \left(  ds\right)  .
\end{array}
\right.
\]

\end{definition}

\noindent Consider now the case when the matrix $H$ depends on the state of
the system. We can reconsider the backward stochastic variational inequality
with oblique reflection in the following manner, $\mathbb{P}-a.s.~\omega
\in\Omega$,%
\begin{equation}
\left\{
\begin{array}
[c]{l}%
Y_{t}+%
%TCIMACRO{\dint _{t}^{T}}%
%BeginExpansion
{\displaystyle\int_{t}^{T}}
%EndExpansion
H\left(  s,Y_{s}\right)  dK_{s}=\eta+%
%TCIMACRO{\dint _{t}^{T}}%
%BeginExpansion
{\displaystyle\int_{t}^{T}}
%EndExpansion
F\left(  s,Y_{s},Z_{s}\right)  ds-\left(  M_{T}-M_{t}\right)  ,\quad\forall
t\in\left[  0,T\right]  ,\smallskip\\
dK_{s}\in\partial\varphi\left(  Y_{s}\right)  \left(  ds\right)  ,
\end{array}
\right.  \label{eq with martingale}%
\end{equation}
where $M$ is a continuous martingale (possible with respect to its natural
filtration if not any other filtration available). If%
\[
H\left(  \omega,t,y\right)  \equiv H\left(  t,y\right)  \quad\text{and}\quad
F\left(  \omega,t,y,z\right)  \equiv F(t,y,z)
\]
we introduce the notion of weak solution of the equation.

\begin{definition}
\label{Def of weak solution}If there exists a probability space $\left(
\Omega,\mathcal{F},\mathbb{P}\right)  $ and a triplet $\left(  Y,M,K\right)
:\Omega\times\left[  0,T\right]  \rightarrow(\mathbb{R}^{d})^{3}$ such that%
\[%
\begin{array}
[c]{cl}%
\begin{array}
[c]{c}%
\left(  a\right)  \bigskip\medskip\\
\end{array}
&
\begin{array}
[c]{l}%
M\text{ is a continuous martingale with respect to the filtration given, for
}\forall t\in\left[  0,T\right]  \text{, by}\\
\mathcal{F}_{t}\overset{def}{=}\mathcal{F}_{t}^{Y,M}=\sigma(\{Y_{s}%
,M_{s}:0\leq s\leq t\})\vee\mathcal{N}\text{,}\smallskip
\end{array}
\\
\left(  b\right)  & Y,K\text{ are c\`{a}dl\`{a}g stochastic processes, adapted
to }\{\mathcal{F}_{t}\}_{t\geq0}\text{,}\smallskip\\
\left(  c\right)  & \text{relation }(\ref{eq with martingale})\text{ is
verified for every }t\in\left[  0,T\right]  \text{, }\mathbb{P}-a.s.~\omega
\in\Omega\text{,}%
\end{array}
\]
the collection $(\Omega,\mathcal{F},\mathbb{P},\mathcal{F}_{t},Y_{t}%
,M_{t},K_{t})_{t\in\left[  0,T\right]  }$ is called a weak solution of the
$BSVI\left(  H\left(  y\right)  ,\varphi,F\right)  $.
\end{definition}

\noindent In both cases given by Definition \ref{Def of strong solution} or
Definition \ref{Def of weak solution} we will say that $(Y,Z,K)$ or $(Y,M,K)$
is a solution of the considered oblique reflected backward stochastic
variational inequality.\medskip

Now we are able to formulate the main results of this article. Denote%
\[
\nu_{t}=%
%TCIMACRO{\dint _{0}^{t}}%
%BeginExpansion
{\displaystyle\int_{0}^{t}}
%EndExpansion
L\left(  s\right)  \left[  \mathbb{E}^{\mathcal{F}_{s}}\left\vert
\eta\right\vert ^{p}\right]  ^{1/p}\quad\text{and}\quad\theta=\sup
_{t\in\left[  0,T\right]  }\left(  \mathbb{E}^{\mathcal{F}_{t}}\left\vert
\eta\right\vert ^{p}\right)  ^{1/p}~.
\]

\begin{theorem}
\label{Th. for strong existence}Let $p>1$ and the assumptions $(H_{1}-H_{4})$
be satisfied, with $l(t)\equiv l<\sqrt{a}$. If%
\begin{equation}
\mathbb{E}e^{\delta\theta}+\mathbb{E}\left\vert \varphi\left(  \eta\right)
\right\vert <\infty\label{boundedness of exp moments}%
\end{equation}
for all $\delta>0$ then the $BSVI\left(  H\left(  t\right)  ,\varphi,F\right)
$ admits a unique strong solution $\left(  Y,Z,K\right)  \in S_{d}^{0}\left[
0,T\right]  \times\Lambda_{d\times k}^{0}\left(  0,T\right)  \times S_{d}%
^{0}\left[  0,T\right]  $ such that, for all $\delta>0$,%
\begin{equation}
\mathbb{E}\sup\limits_{s\in\left[  0,T\right]  }e^{\delta p\nu_{s}}\left\vert
Y_{s}\right\vert ^{p}+\mathbb{E}\left(
%TCIMACRO{\dint _{0}^{T}}%
%BeginExpansion
{\displaystyle\int_{0}^{T}}
%EndExpansion
e^{2\delta\nu_{s}}\left\vert Z_{s}\right\vert ^{2}ds\right)  ^{p/2}<\infty.
\label{boundedness with weight}%
\end{equation}
Moreover, there exists a positive constant, independent of the terminal time
$T$, $C=C(a,b,\Lambda)$ such that, $\mathbb{P}-a.s.~\omega\in\Omega$,%
\[
\left\vert Y_{t}\right\vert \leq C\left(  1+\left[  \mathbb{E}^{\mathcal{F}%
_{t}}\left\vert \eta\right\vert ^{p}\right]  ^{1/p}\right)  ,\quad~\text{for
all }t\in\left[  0,T\right]
\]
and the process $K$ can be represented as%
\[
K_{t}=%
%TCIMACRO{\dint _{0}^{t}}%
%BeginExpansion
{\displaystyle\int_{0}^{t}}
%EndExpansion
U_{s}ds,
\]
where%
\[
\mathbb{E}\int_{0}^{T}|U_{t}|^{2}dt+\mathbb{E}\int_{0}^{T}|Z_{t}|^{2}dt\leq
C\left(  \mathbb{E}|\eta|^{2}+\mathbb{E}\left\vert \varphi\left(  \eta\right)
\right\vert +\mathbb{E}\int_{0}^{T}|F(t,0,0)|^{2}dt\right)  .
\]

\end{theorem}

\begin{remark}
The boundedness conditions imposed to the exponential moments from
(\ref{boundedness of exp moments}) is not a very restrictive one. For example,
it takes place if we consider $k=1$ and $\eta=B_{T}^{\alpha}$, with
$0<\alpha<2$.
\end{remark}

\begin{theorem}
\label{Th. for weak existence}Let the assumptions $(H_{2}-H_{4})$ be
satisfied. Then the $BSVI\left(  H\left(  t,y\right)  ,\varphi,F\right)  $
(\ref{main oblique BSVI}) admits a unique weak solution $(\Omega
,\mathcal{F},\mathbb{P},\mathcal{F}_{t},B_{t},Y_{t},M_{t},K_{t})_{t\in\left[
0,T\right]  }$.
\end{theorem}

The proofs of the above results are detailed in the next sections. Section 4
deals with a sequence of approximating equations and apriori estimates of
their solutions. The estimates will be valid for both cases covered by Theorem
\ref{Th. for strong existence} and Theorem \ref{Th. for weak existence}. After
this, the proof is split in Section 5 and Section 6, each one being dedicated
to the particularities brought by Theorem \ref{Th. for strong existence} and
Theorem \ref{Th. for weak existence}.

\section{Approximating problems and apriori estimates}

In order to prove the existence of the solution (strong or weak) we can
assume, without loosing the generality, that%
\[
\varphi\left(  y\right)  \geq\varphi\left(  0\right)  =0
\]
because, otherwise, we can change the functions $\varphi$, $F$ and $H$ as
follows%
\begin{align*}
\tilde{\varphi}\left(  y\right)   &  =\varphi\left(  y+u_{0}\right)
-\varphi\left(  u_{0}\right)  -\left\langle \hat{u}_{0},y\right\rangle
\geq0,\smallskip\\
\tilde{F}\left(  t,y,z\right)   &  =F\left(  t,y+u_{0},z\right)  -H\left(
t,y+u_{0}\right)  \hat{u}_{0},\smallskip\\
\tilde{H}\left(  t,y\right)   &  =H\left(  t,y+u_{0}\right)  ,
\end{align*}
with $u_{0}\in Dom\left(  \partial\varphi\right)  $ and $\hat{u}_{0}%
\in\partial\varphi\left(  u_{0}\right)  $. The solution is now given by
$\left(  Y,Z,K\right)  =(\tilde{Y}+u_{0},\tilde{Z},\tilde{K})$, where%
\[
\left\{
\begin{array}
[c]{l}%
\tilde{Y}_{t}+%
%TCIMACRO{\dint _{t}^{T}}%
%BeginExpansion
{\displaystyle\int_{t}^{T}}
%EndExpansion
\tilde{H}(s,\tilde{Y}_{s})d\tilde{K}_{s}=\left(  \eta-u_{0}\right)  +%
%TCIMACRO{\dint _{t}^{T}}%
%BeginExpansion
{\displaystyle\int_{t}^{T}}
%EndExpansion
F(s,\tilde{Y}_{s},\tilde{Z}_{s})ds-%
%TCIMACRO{\dint _{t}^{T}}%
%BeginExpansion
{\displaystyle\int_{t}^{T}}
%EndExpansion
\tilde{Z}_{s}dB_{s},\text{ }\forall t\in\left[  0,T\right]  ,\smallskip\\
d\tilde{K}_{s}\left(  \omega\right)  \in\partial\tilde{\varphi}(\tilde{Y}%
_{s}\left(  \omega\right)  )\left(  ds\right)  ,\text{ }\forall s,\text{
}\mathbb{P}-a.s.~\omega\in\Omega.
\end{array}
\right.
\]
We start simultaneously the proofs of Theorem \ref{Th. for strong existence}
and Theorem \ref{Th. for weak existence} by obtaining some apriori estimates
for the solutions of the approximating equations.\medskip

\begin{proof}
Let $p>1$.

\noindent\textbf{Step 1.} \textit{Boundedness under the assumption}%
\[
0\leq\ell\left(  t\right)  \equiv\ell<\sqrt{a}.
\]

\noindent Let $0<\varepsilon\leq1.$ Consider the approximating BSDE%
\begin{equation}
Y_{t}^{\varepsilon}+%
%TCIMACRO{\dint _{t}^{T}}%
%BeginExpansion
{\displaystyle\int_{t}^{T}}
%EndExpansion
H\left(  s,Y_{s}^{\varepsilon}\right)  \nabla\varphi_{\varepsilon}\left(
Y_{s}^{\varepsilon}\right)  ds=\eta+%
%TCIMACRO{\dint _{t}^{T}}%
%BeginExpansion
{\displaystyle\int_{t}^{T}}
%EndExpansion
F\left(  s,Y_{s}^{\varepsilon},Z_{s}^{\varepsilon}\right)  ds-%
%TCIMACRO{\dint _{t}^{T}}%
%BeginExpansion
{\displaystyle\int_{t}^{T}}
%EndExpansion
Z_{s}^{\varepsilon}dB_{s},\quad\forall t\in\left[  0,T\right]  .
\label{approximating eq for general case}%
\end{equation}
Let $\tilde{F}\left(  t,y\right)  =F\left(  t,y,z\right)  -H\left(
t,y\right)  \nabla\varphi_{\varepsilon}\left(  y\right)  $. Using the
Lipschitz and boundedness hypothesis imposed on $F$ and $H$, we have, for all
$t\in\left[  0,T\right]  $, $y,y^{\prime}\in\mathbb{R}^{d}$, $z,z^{\prime}%
\in\mathbb{R}^{d\times k}$,%
\begin{align*}
&  |\tilde{F}(t,y^{\prime},z^{\prime})-\tilde{F}(t,y,z)|\\
&  \leq|F(t,y^{\prime},z^{\prime})-F(t,y,z)|+|\left[  H\left(  t,y\right)
-H\left(  t,y^{\prime}\right)  \right]  \nabla\varphi_{\varepsilon}\left(
y\right)  |+|H\left(  t,y^{\prime}\right)  \left[  \nabla\varphi_{\varepsilon
}\left(  y\right)  -\nabla\varphi_{\varepsilon}\left(  y^{\prime}\right)
\right]  |\\
&  \leq L\left(  t\right)  |y^{\prime}-y|+\ell|z^{\prime}-z|+\dfrac{\Lambda
}{\varepsilon}|y^{\prime}-y|\left\vert y\right\vert +\dfrac{b}{\varepsilon
}|y^{\prime}-y|\\
&  \leq\left(  L\left(  t\right)  +\frac{\Lambda}{\varepsilon}+\frac
{b}{\varepsilon}\right)  (1+\left\vert y\right\vert \vee|y^{\prime
}|)|y^{\prime}-y|+\ell|z^{\prime}-z|
\end{align*}
and%
\[
|\tilde{F}(t,y,0)|\leq\rho\left(  t\right)  +L\left(  t\right)  \left\vert
y\right\vert +\dfrac{b}{\varepsilon}\left\vert y\right\vert .
\]
By Theorem \ref{Corollary existence sol} (see Annex 6.1), the BSDE
(\ref{approximating eq for general case}) has a unique solution
$(Y^{\varepsilon},Z^{\varepsilon})\in S_{d}^{0}\left[  0,T\right]
\times\Lambda_{d\times k}^{0}\left(  0,T\right)  $ such that for all
$\delta>0$,%
\begin{equation}
\mathbb{E}\sup\limits_{s\in\left[  0,T\right]  }e^{\delta p\nu_{s}}\left\vert
Y_{s}^{\varepsilon}\right\vert ^{p}+\mathbb{E}\left(
%TCIMACRO{\dint _{0}^{T}}%
%BeginExpansion
{\displaystyle\int_{0}^{T}}
%EndExpansion
e^{2\delta\nu_{s}}\left\vert Z_{s}^{\varepsilon}\right\vert ^{2}ds\right)
^{p/2}<\infty. \label{space of sol for approximating equation}%
\end{equation}
By Energy Equality we obtain%
\begin{align*}
|Y_{t}^{\varepsilon}|^{2}+2%
%TCIMACRO{\dint _{t}^{s}}%
%BeginExpansion
{\displaystyle\int_{t}^{s}}
%EndExpansion
\left\langle Y_{r}^{\varepsilon},H\left(  r,Y_{r}^{\varepsilon}\right)
\nabla\varphi_{\varepsilon}\left(  Y_{r}^{\varepsilon}\right)  \right\rangle
dr+%
%TCIMACRO{\dint _{t}^{s}}%
%BeginExpansion
{\displaystyle\int_{t}^{s}}
%EndExpansion
|Z_{r}^{\varepsilon}|^{2}dr  &  =|Y_{s}^{\varepsilon}|^{2}+2%
%TCIMACRO{\dint _{t}^{s}}%
%BeginExpansion
{\displaystyle\int_{t}^{s}}
%EndExpansion
\left\langle Y_{r}^{\varepsilon},F\left(  r,Y_{r}^{\varepsilon},Z_{r}%
^{\varepsilon}\right)  \right\rangle dr\\
&  -2%
%TCIMACRO{\dint _{t}^{s}}%
%BeginExpansion
{\displaystyle\int_{t}^{s}}
%EndExpansion
\left\langle Y_{r}^{\varepsilon},Z_{r}^{\varepsilon}dB_{r}\right\rangle .
\end{align*}
Since $y\longmapsto\varphi_{\varepsilon}\left(  y\right)  :\mathbb{R}%
^{d}\rightarrow\mathbb{R}$ is a convex $C^{1}-$function, then by the
subdifferential inequality (\ref{subdiff ineq 1}) (see Annex 6.3)%
\begin{align*}
&  \varphi_{\varepsilon}\left(  Y_{t}^{\varepsilon}\right)  +%
%TCIMACRO{\dint _{t}^{s}}%
%BeginExpansion
{\displaystyle\int_{t}^{s}}
%EndExpansion
\left\langle \nabla\varphi_{\varepsilon}\left(  Y_{r}^{\varepsilon}\right)
,H(r,Y_{r}^{\varepsilon})\nabla\varphi_{\varepsilon}\left(  Y_{r}%
^{\varepsilon}\right)  \right\rangle dr\\
&  \leq\varphi_{\varepsilon}\left(  Y_{s}^{\varepsilon}\right)  +%
%TCIMACRO{\dint _{t}^{s}}%
%BeginExpansion
{\displaystyle\int_{t}^{s}}
%EndExpansion
\left\langle \nabla\varphi_{\varepsilon}(Y_{r}^{\varepsilon}),F(r,Y_{r}%
^{\varepsilon},Z_{r}^{\varepsilon})\right\rangle dr-%
%TCIMACRO{\dint _{t}^{s}}%
%BeginExpansion
{\displaystyle\int_{t}^{s}}
%EndExpansion
\left\langle \nabla\varphi_{\varepsilon}(Y_{r}^{\varepsilon}),Z_{r}%
^{\varepsilon}dB_{r}\right\rangle .
\end{align*}
As consequence, combining the previous two inequalities, we obtain%
\begin{align}
&  |Y_{t}^{\varepsilon}|^{2}+\varphi_{\varepsilon}(Y_{t}^{\varepsilon})+%
%TCIMACRO{\dint _{t}^{s}}%
%BeginExpansion
{\displaystyle\int_{t}^{s}}
%EndExpansion
\left\langle \nabla\varphi_{\varepsilon}(Y_{r}^{\varepsilon}),H(r,Y_{r}%
^{\varepsilon})\nabla\varphi_{\varepsilon}(Y_{r}^{\varepsilon})\right\rangle
dr+%
%TCIMACRO{\dint _{t}^{s}}%
%BeginExpansion
{\displaystyle\int_{t}^{s}}
%EndExpansion
|Z_{r}^{\varepsilon}|^{2}dr\label{ineq2}\\
&  \leq|Y_{s}^{\varepsilon}|^{2}+\varphi_{\varepsilon}(Y_{s}^{\varepsilon})+2%
%TCIMACRO{\dint _{t}^{s}}%
%BeginExpansion
{\displaystyle\int_{t}^{s}}
%EndExpansion
\left\langle Y_{r}^{\varepsilon},F(r,Y_{r}^{\varepsilon},Z_{r}^{\varepsilon
})\right\rangle dr\nonumber\\
&  +%
%TCIMACRO{\dint _{t}^{s}}%
%BeginExpansion
{\displaystyle\int_{t}^{s}}
%EndExpansion
\left\langle \nabla\varphi_{\varepsilon}(Y_{r}^{\varepsilon}),F(r,Y_{r}%
^{\varepsilon},Z_{r}^{\varepsilon})-2H(r,Y_{r}^{\varepsilon})^{\ast}%
Y_{r}^{\varepsilon}\right\rangle dr-%
%TCIMACRO{\dint _{t}^{s}}%
%BeginExpansion
{\displaystyle\int_{t}^{s}}
%EndExpansion
\left\langle 2Y_{r}^{\varepsilon}+\nabla\varphi_{\varepsilon}(Y_{r}%
^{\varepsilon}),Z_{r}^{\varepsilon}dB_{r}\right\rangle .\nonumber
\end{align}
Let $\lambda>0$. In the sequel we denote by $C$ a generic positive constant,
independent of $\varepsilon,\delta\in(0,1]$, constant which can change from
one line to another, without affecting the result. The assumptions $(H_{2})$
and (\ref{hypothesis on F}) lead to the following estimates:

\begin{itemize}
\item $\quad\left\langle \nabla\varphi_{\varepsilon}\left(  Y_{r}%
^{\varepsilon}\right)  ,H\left(  r,Y_{r}^{\varepsilon}\right)  \nabla
\varphi_{\varepsilon}\left(  Y_{r}^{\varepsilon}\right)  \right\rangle \geq
a\left\vert \nabla\varphi_{\varepsilon}\left(  Y_{r}^{\varepsilon}\right)
\right\vert ^{2}$\medskip

\item $\quad2\left\langle Y_{r}^{\varepsilon},F\left(  r,Y_{r}^{\varepsilon
},Z_{r}^{\varepsilon}\right)  \right\rangle \leq2\ell\left\vert Y_{r}%
^{\varepsilon}\right\vert \left\vert Z_{r}^{\varepsilon}\right\vert +2L\left(
r\right)  \left\vert Y_{r}^{\varepsilon}\right\vert ^{2}+2\left\vert
Y_{r}^{\varepsilon}\right\vert \left\vert F\left(  r,0,0\right)  \right\vert
$\medskip

$\quad\leq\lambda\left\vert Z_{r}^{\varepsilon}\right\vert ^{2}+\left(
2L\left(  r\right)  +\dfrac{\ell^{2}}{\lambda}+1\right)  \left\vert
Y_{r}^{\varepsilon}\right\vert ^{2}+\rho^{2}\left(  r\right)  $\medskip

\item $\quad\left\langle \nabla\varphi_{\varepsilon}\left(  Y_{r}%
^{\varepsilon}\right)  ,F\left(  r,Y_{r}^{\varepsilon},Z_{r}^{\varepsilon
}\right)  -2H\left(  r,Y_{r}^{\varepsilon}\right)  ^{\ast}Y_{r}^{\varepsilon
}\right\rangle $\medskip

$\quad\leq\left\vert \nabla\varphi_{\varepsilon}\left(  Y_{r}^{\varepsilon
}\right)  \right\vert \left[  \ell\left\vert Z_{r}^{\varepsilon}\right\vert
+L\left(  r\right)  \left\vert Y_{r}^{\varepsilon}\right\vert +\left\vert
F\left(  r,0,0\right)  \right\vert +2b\left\vert Y_{r}^{\varepsilon
}\right\vert \right]  $\medskip

$\quad\leq\lambda\left\vert Z_{r}^{\varepsilon}\right\vert ^{2}+\dfrac
{1}{4\lambda}\ell^{2}\left\vert \nabla\varphi_{\varepsilon}\left(
Y_{r}^{\varepsilon}\right)  \right\vert ^{2}+\dfrac{a}{4\lambda}\left\vert
\nabla\varphi_{\varepsilon}\left(  Y_{r}^{\varepsilon}\right)  \right\vert
^{2}+\dfrac{2\lambda}{a}\left[  \left(  L\left(  r\right)  +2b\right)
^{2}\left\vert Y_{r}^{\varepsilon}\right\vert ^{2}+\left\vert F\left(
r,0,0\right)  \right\vert ^{2}\right]  $\medskip
\end{itemize}

\noindent Inserting the above estimates in (\ref{ineq2}), we obtain,
$\mathbb{P}-a.s.,$ for all $0\leq t\leq s\leq T$,%
\[%
\begin{array}
[c]{l}%
\left\vert Y_{t}^{\varepsilon}\right\vert ^{2}+\varphi_{\varepsilon}\left(
Y_{t}^{\varepsilon}\right)  +\left(  a-\dfrac{a+\ell^{2}}{4\lambda}\right)
%TCIMACRO{\dint _{t}^{s}}%
%BeginExpansion
{\displaystyle\int_{t}^{s}}
%EndExpansion
\left\vert \nabla\varphi_{\varepsilon}\left(  Y_{r}^{\varepsilon}\right)
\right\vert ^{2}dr+\left(  1-2\lambda\right)
%TCIMACRO{\dint _{t}^{s}}%
%BeginExpansion
{\displaystyle\int_{t}^{s}}
%EndExpansion
\left\vert Z_{r}^{\varepsilon}\right\vert ^{2}dr\medskip\\
\quad\quad\quad\leq\left\vert Y_{s}^{\varepsilon}\right\vert ^{2}%
+\varphi_{\varepsilon}\left(  Y_{s}^{\varepsilon}\right)  +%
%TCIMACRO{\dint _{t}^{s}}%
%BeginExpansion
{\displaystyle\int_{t}^{s}}
%EndExpansion
\left(  1+\dfrac{2\lambda}{a}\right)  \left\vert F\left(  r,0,0\right)
\right\vert ^{2}dr\medskip\\
\quad\quad\quad+%
%TCIMACRO{\dint _{t}^{s}}%
%BeginExpansion
{\displaystyle\int_{t}^{s}}
%EndExpansion
\left(  2L\left(  r\right)  +\dfrac{\ell^{2}}{\lambda}+1+\dfrac{2\lambda}%
{a}\left(  L\left(  r\right)  +2b\right)  ^{2}\right)  \left\vert
Y_{r}^{\varepsilon}\right\vert ^{2}dr-%
%TCIMACRO{\dint _{t}^{s}}%
%BeginExpansion
{\displaystyle\int_{t}^{s}}
%EndExpansion
\left\langle 2Y_{r}^{\varepsilon}+\nabla\varphi_{\varepsilon}\left(
Y_{r}^{\varepsilon}\right)  ,Z_{r}^{\varepsilon}dB_{r}\right\rangle .
\end{array}
\]
Denote%
\[
K_{t}^{\lambda}=%
%TCIMACRO{\dint _{0}^{t}}%
%BeginExpansion
{\displaystyle\int_{0}^{t}}
%EndExpansion
\left[  \left(  1+\dfrac{2\lambda}{a}\right)  \left\vert F\left(
r,0,0\right)  \right\vert ^{2}-\left(  a-\dfrac{a+\ell^{2}}{4\lambda}\right)
\left\vert \nabla\varphi_{\varepsilon}\left(  Y_{r}^{\varepsilon}\right)
\right\vert ^{2}-\left(  1-2\lambda\right)  \left\vert Z_{r}^{\varepsilon
}\right\vert ^{2}\right]  dr
\]
and%
\[
A\left(  t\right)  =%
%TCIMACRO{\dint _{0}^{t}}%
%BeginExpansion
{\displaystyle\int_{0}^{t}}
%EndExpansion
\left(  2L\left(  r\right)  +\dfrac{\ell^{2}}{\lambda}+1+\dfrac{2\lambda}%
{a}\left(  L\left(  r\right)  +2b\right)  ^{2}\right)  dr.
\]
Since $\varphi_{\varepsilon}\left(  y\right)  \geq\varphi_{\varepsilon}\left(
0\right)  =0$ we have%
\[%
\begin{array}
[c]{c}%
\left\vert Y_{t}^{\varepsilon}\right\vert ^{2}+\varphi_{\varepsilon}\left(
Y_{t}^{\varepsilon}\right)  \leq\left\vert Y_{s}^{\varepsilon}\right\vert
^{2}+\varphi_{\varepsilon}\left(  Y_{s}^{\varepsilon}\right)  +%
%TCIMACRO{\dint _{t}^{s}}%
%BeginExpansion
{\displaystyle\int_{t}^{s}}
%EndExpansion
\left[  dK_{r}^{\lambda}+\left[  \left\vert Y_{r}^{\varepsilon}\right\vert
^{2}+\varphi_{\varepsilon}\left(  Y_{r}^{\varepsilon}\right)  \right]
dA\left(  r\right)  \right]  \medskip\\
-%
%TCIMACRO{\dint _{t}^{s}}%
%BeginExpansion
{\displaystyle\int_{t}^{s}}
%EndExpansion
\left\langle 2Y_{r}^{\varepsilon}+\nabla\varphi_{\varepsilon}\left(
Y_{r}^{\varepsilon}\right)  ,Z_{r}^{\varepsilon}dB_{r}\right\rangle
\end{array}
\]
and, by Proposition \ref{exp prop ineq} (see Annex 6.3), we infer%
\begin{equation}%
\begin{array}
[c]{r}%
e^{A\left(  t\right)  }\left(  \left\vert Y_{t}^{\varepsilon}\right\vert
^{2}+\varphi_{\varepsilon}\left(  Y_{t}^{\varepsilon}\right)  \right)  \leq
e^{A\left(  s\right)  }\left[  \left\vert Y_{s}^{\varepsilon}\right\vert
^{2}+\varphi_{\varepsilon}\left(  Y_{s}^{\varepsilon}\right)  \right]  +%
%TCIMACRO{\dint _{t}^{s}}%
%BeginExpansion
{\displaystyle\int_{t}^{s}}
%EndExpansion
e^{A\left(  r\right)  }dK_{r}^{\lambda}\medskip\\
-%
%TCIMACRO{\dint _{t}^{s}}%
%BeginExpansion
{\displaystyle\int_{t}^{s}}
%EndExpansion
e^{A\left(  r\right)  }\left\langle 2Y_{r}^{\varepsilon}+\nabla\varphi
_{\varepsilon}\left(  Y_{r}^{\varepsilon}\right)  ,Z_{r}^{\varepsilon}%
dB_{r}\right\rangle .
\end{array}
\label{ineq3}%
\end{equation}
Let $\lambda=\dfrac{1}{2}\left(  \dfrac{a+\ell^{2}}{4a}+\dfrac{1}{2}\right)  $
be fixed. It follows that, for all $0\leq t\leq s\leq T$,%
\begin{equation}%
\begin{array}
[c]{l}%
\left\vert Y_{t}^{\varepsilon}\right\vert ^{2}+\varphi_{\varepsilon}\left(
Y_{t}^{\varepsilon}\right)  +\mathbb{E}^{\mathcal{F}_{t}}%
%TCIMACRO{\dint _{t}^{s}}%
%BeginExpansion
{\displaystyle\int_{t}^{s}}
%EndExpansion
\left\vert \nabla\varphi_{\varepsilon}\left(  Y_{r}^{\varepsilon}\right)
\right\vert ^{2}dr+\mathbb{E}^{\mathcal{F}_{t}}%
%TCIMACRO{\dint _{t}^{s}}%
%BeginExpansion
{\displaystyle\int_{t}^{s}}
%EndExpansion
\left\vert Z_{r}^{\varepsilon}\right\vert ^{2}dr\medskip\\
\quad\quad\quad\leq C\mathbb{E}^{\mathcal{F}_{t}}\left\vert Y_{s}%
^{\varepsilon}\right\vert ^{2}+\mathbb{E}^{\mathcal{F}_{t}}\varphi
_{\varepsilon}\left(  Y_{s}^{\varepsilon}\right)  +C\mathbb{E}^{\mathcal{F}%
_{t}}%
%TCIMACRO{\dint _{t}^{s}}%
%BeginExpansion
{\displaystyle\int_{t}^{s}}
%EndExpansion
\left\vert F\left(  r,0,0\right)  \right\vert ^{2}dr.
\end{array}
\label{ineq4}%
\end{equation}
In particular, we consider $s=T$ and, since $0\leq\varphi_{\varepsilon}\left(
\eta\right)  \leq\varphi\left(  \eta\right)  $,%
\begin{equation}
\mathbb{E}%
%TCIMACRO{\dint _{0}^{T}}%
%BeginExpansion
{\displaystyle\int_{0}^{T}}
%EndExpansion
\left\vert \nabla\varphi_{\varepsilon}\left(  Y_{r}^{\varepsilon}\right)
\right\vert ^{2}dr+\mathbb{E}%
%TCIMACRO{\dint _{0}^{T}}%
%BeginExpansion
{\displaystyle\int_{0}^{T}}
%EndExpansion
\left\vert Z_{r}^{\varepsilon}\right\vert ^{2}dr\leq C\left[  \mathbb{E}%
\left\vert \eta\right\vert ^{2}+\mathbb{E}\varphi\left(  \eta\right)
+\mathbb{E}%
%TCIMACRO{\dint _{0}^{T}}%
%BeginExpansion
{\displaystyle\int_{0}^{T}}
%EndExpansion
\left\vert F\left(  r,0,0\right)  \right\vert ^{2}dr\right]  =\tilde{C}.
\label{ineq5}%
\end{equation}
Using the definition of $\nabla\varphi_{\varepsilon}$ we also obtain%
\begin{equation}
\mathbb{E}%
%TCIMACRO{\dint _{0}^{T}}%
%BeginExpansion
{\displaystyle\int_{0}^{T}}
%EndExpansion
\left\vert Y_{r}^{\varepsilon}-J_{\varepsilon}(Y_{r}^{\varepsilon})\right\vert
^{2}dr\leq\tilde{C}\varepsilon^{2}. \label{ineq6}%
\end{equation}

\noindent We write the approximating BSDE
(\ref{approximating eq for general case}) under the form%
\[
Y_{t}^{\varepsilon}=\eta+%
%TCIMACRO{\dint _{t}^{T}}%
%BeginExpansion
{\displaystyle\int_{t}^{T}}
%EndExpansion
dK_{s}^{\varepsilon}-%
%TCIMACRO{\dint _{t}^{T}}%
%BeginExpansion
{\displaystyle\int_{t}^{T}}
%EndExpansion
Z_{s}^{\varepsilon}dB_{s},
\]
where%
\[
dK_{s}^{\varepsilon}=\left[  F\left(  s,Y_{s}^{\varepsilon},Z_{s}%
^{\varepsilon}\right)  -H\left(  s,Y_{s}^{\varepsilon}\right)  \nabla
\varphi_{\varepsilon}\left(  Y_{s}^{\varepsilon}\right)  \right]  ds.
\]

\noindent If we denote%
\[
N_{t}=%
%TCIMACRO{\dint _{0}^{t}}%
%BeginExpansion
{\displaystyle\int_{0}^{t}}
%EndExpansion
\left[  \left\vert F\left(  s,0,0\right)  \right\vert +b\left\vert
\nabla\varphi_{\varepsilon}\left(  Y_{s}^{\varepsilon}\right)  \right\vert
\right]  ds\quad\text{and}\quad V\left(  t\right)  =%
%TCIMACRO{\dint _{0}^{t}}%
%BeginExpansion
{\displaystyle\int_{0}^{t}}
%EndExpansion
\left(  L\left(  s\right)  +\ell^{2}\right)  ds,
\]
then%
\[
\left\langle Y_{t}^{\varepsilon},dK_{t}^{\varepsilon}\right\rangle
\leq\left\vert Y_{t}^{\varepsilon}\right\vert dN_{t}+\left[  \left\vert
F\left(  t,0,0\right)  \right\vert +b\left\vert \nabla\varphi_{\varepsilon
}\left(  Y_{t}^{\varepsilon}\right)  \right\vert \right]  dt+\left\vert
Y_{t}^{\varepsilon}\right\vert ^{2}dV\left(  t\right)  +\dfrac{1}{4}\left\vert
Z_{t}^{\varepsilon}\right\vert ^{2}dt.
\]
We apply Proposition \ref{ineq cond exp} (see Annex 6.3) and it infers, for
$p=2$,%

\[%
\begin{array}
[c]{l}%
\mathbb{E}^{\mathcal{F}_{t}}\sup\limits_{s\in\left[  t,T\right]  }\left\vert
e^{V\left(  s\right)  }Y_{s}^{\varepsilon}\right\vert ^{2}+\mathbb{E}%
^{\mathcal{F}_{t}}\left(
%TCIMACRO{\dint _{t}^{T}}%
%BeginExpansion
{\displaystyle\int_{t}^{T}}
%EndExpansion
e^{2V\left(  s\right)  }\left\vert Z_{s}^{\varepsilon}\right\vert
^{2}ds\right)  \medskip\\
\quad\quad\quad\leq C\mathbb{E}^{\mathcal{F}_{t}}\left[  \left\vert
e^{V\left(  T\right)  }\eta\right\vert ^{2}+\left(
%TCIMACRO{\dint _{t}^{T}}%
%BeginExpansion
{\displaystyle\int_{t}^{T}}
%EndExpansion
e^{V\left(  s\right)  }\left[  \left\vert F\left(  s,0,0\right)  \right\vert
+b\left\vert \nabla\varphi_{\varepsilon}\left(  Y_{s}^{\varepsilon}\right)
\right\vert \right]  ds\right)  ^{2}\right]  .
\end{array}
\]
Taking into account (\ref{ineq5}) it follows%
\begin{equation}
\left\vert Y_{0}^{\varepsilon}\right\vert ^{2}\leq\mathbb{E}\sup
\limits_{s\in\left[  0,T\right]  }\left\vert Y_{s}^{\varepsilon}\right\vert
^{2}\leq C\left[  \mathbb{E}\left\vert \eta\right\vert ^{2}+\mathbb{E}%
\varphi\left(  \eta\right)  +\mathbb{E}%
%TCIMACRO{\dint _{0}^{T}}%
%BeginExpansion
{\displaystyle\int_{0}^{T}}
%EndExpansion
\left\vert F\left(  r,0,0\right)  \right\vert ^{2}dr\right]  . \label{ineq8}%
\end{equation}
\noindent The Lipschitz and the boundedness hypotheses $\left(  H_{4}\right)
$ imposed on $F$ lead, due to the fact that $l$ is constant, $L\in
L^{2}\left(  0,T;\mathbb{R}_{+}\right)  $ and $\rho\in L^{1}\left(
0,T;\mathbb{R}_{+}\right)  $, to%
\[%
\begin{array}
[c]{l}%
\mathbb{E}%
%TCIMACRO{\dint _{0}^{T}}%
%BeginExpansion
{\displaystyle\int_{0}^{T}}
%EndExpansion
|F(r,Y_{r}^{\varepsilon},Z_{r}^{\varepsilon})|^{2}dr\leq2\mathbb{E}%
%TCIMACRO{\dint _{0}^{T}}%
%BeginExpansion
{\displaystyle\int_{0}^{T}}
%EndExpansion
|F(r,Y_{r}^{\varepsilon},Z_{r}^{\varepsilon})-F(r,Y_{r}^{\varepsilon}%
,0)|^{2}dr+2\mathbb{E}%
%TCIMACRO{\dint _{0}^{T}}%
%BeginExpansion
{\displaystyle\int_{0}^{T}}
%EndExpansion
|F(r,Y_{r}^{\varepsilon},0)|^{2}dr\medskip\\
\quad\quad\quad\quad\quad\quad\quad\quad\quad\leq2l^{2}\mathbb{E}%
%TCIMACRO{\dint _{0}^{T}}%
%BeginExpansion
{\displaystyle\int_{0}^{T}}
%EndExpansion
|Z_{r}^{\varepsilon}|^{2}dr+4\mathbb{E}%
%TCIMACRO{\dint _{0}^{T}}%
%BeginExpansion
{\displaystyle\int_{0}^{T}}
%EndExpansion
L^{2}(r)|Y_{r}^{\varepsilon}|^{2}dr+4\mathbb{E}%
%TCIMACRO{\dint _{0}^{T}}%
%BeginExpansion
{\displaystyle\int_{0}^{T}}
%EndExpansion
|F(r,0,0)|^{2}dr\leq C,
\end{array}
\]%
\begin{equation}
\mathbb{E}%
%TCIMACRO{\dint _{0}^{T}}%
%BeginExpansion
{\displaystyle\int_{0}^{T}}
%EndExpansion
|F(r,Y_{r}^{\varepsilon},Z_{r}^{\varepsilon})|dr\leq\mathbb{E}%
%TCIMACRO{\dint _{0}^{T}}%
%BeginExpansion
{\displaystyle\int_{0}^{T}}
%EndExpansion
\left[  L\left(  r\right)  |Y_{r}^{\varepsilon}|+l|Z_{r}^{\varepsilon
}|+|F(r,0,0)|\right]  dr\leq C \label{ineq7}%
\end{equation}
For the convenience of the reader, we will group together, under the form of a
Lemma, some useful estimations on the solution of the approximating equation,
estimation that we just obtained in Step 1.

\begin{lemma}
\label{Lemma with the estimations from Step 1}Consider the approximating BSDE
(\ref{approximating eq for general case}), with its solution $\left(
Y^{\varepsilon},Z^{\varepsilon}\right)  $ and denote $U^{\varepsilon}%
=\nabla\varphi_{\varepsilon}(Y^{\varepsilon})$. There exists a positive
constant $C=C(a,b,\Lambda,l,L(\cdot))$, independent of $\varepsilon$, such
that%
\begin{equation}
\mathbb{E}\sup\limits_{s\in\left[  0,T\right]  }\left\vert Y_{s}^{\varepsilon
}\right\vert ^{2}+\mathbb{E}%
%TCIMACRO{\dint _{0}^{T}}%
%BeginExpansion
{\displaystyle\int_{0}^{T}}
%EndExpansion
(\left\vert U_{r}^{\varepsilon}\right\vert ^{2}+\left\vert Z_{r}^{\varepsilon
}\right\vert ^{2})dr\leq C\left[  \mathbb{E}\left\vert \eta\right\vert
^{2}+\mathbb{E}\varphi\left(  \eta\right)  +\mathbb{E}%
%TCIMACRO{\dint _{0}^{T}}%
%BeginExpansion
{\displaystyle\int_{0}^{T}}
%EndExpansion
\left\vert F\left(  r,0,0\right)  \right\vert ^{2}dr\right]  .
\label{ineq Y,Z,U}%
\end{equation}
$\smallskip$
\end{lemma}

\noindent\textbf{Step 2.} \textit{Convergences under the assumption}%
\[
0\leq\ell\left(  t\right)  \equiv\ell<\sqrt{a}.
\]
The estimations of Step 1 imply that there exist a sequence $\left\{
\varepsilon_{n}:n\in\mathbb{N}^{\ast}\right\}  ,$ $\varepsilon_{n}%
\rightarrow0$ as $n\rightarrow\infty$, and six progressively measurable
stochastic processes $Y,Z,U,F,\chi,h$ such that%
\begin{align*}
Y_{0}^{\varepsilon_{n}}  &  \rightarrow Y_{0},\quad\text{in }\mathbb{R}^{d},\\
Z^{\varepsilon_{n}}  &  \rightharpoonup Z,\quad\text{weakly in }L^{2}%
(\Omega\times\left(  0,T\right)  ;\mathbb{R}^{d\times k}),
\end{align*}
and, weakly in $L^{2}(\Omega\times\left(  0,T\right)  ;\mathbb{R}^{d})$,%
\[%
\begin{array}
[c]{c}%
Y^{\varepsilon_{n}}\rightharpoonup Y,\quad\quad\nabla\varphi_{\varepsilon_{n}%
}\left(  Y^{\varepsilon_{n}}\right)  \rightharpoonup U,\quad\quad H\left(
\cdot,Y^{\varepsilon_{n}}\right)  \rightharpoonup h,\medskip\\
H\left(  \cdot,Y^{\varepsilon_{n}}\right)  \nabla\varphi_{\varepsilon_{n}%
}\left(  Y^{\varepsilon_{n}}\right)  \rightharpoonup\chi\quad\quad
\text{and}\quad\quad F\left(  \cdot,Y^{\varepsilon_{n}},Z^{\varepsilon_{n}%
}\right)  \rightharpoonup F.\medskip
\end{array}
\]
The convergence $Y^{\varepsilon_{n}}\rightharpoonup Y$ and the inequality
(\ref{ineq6}), written for $\varepsilon=\varepsilon_{n}$, imply that, on the
sequence $\left\{  \varepsilon_{n}:n\in\mathbb{N}^{\ast}\right\}  $,%
\[
J_{\varepsilon_{n}}(Y^{\varepsilon_{n}})\rightharpoonup Y,\quad\text{weakly in
}L^{2}(\Omega\times\left(  0,T\right)  ;\mathbb{R}^{d})\text{.}%
\]
We write (\ref{space of sol for approximating equation}) for $\varepsilon
=\varepsilon_{n}$ and, passing to $\liminf_{n\rightarrow+\infty}$, we obtain%
\[
\mathbb{E}\sup\limits_{s\in\left[  0,T\right]  }e^{\delta p\nu_{s}}\left\vert
Y_{s}\right\vert ^{p}+\mathbb{E}\left(
%TCIMACRO{\dint _{0}^{T}}%
%BeginExpansion
{\displaystyle\int_{0}^{T}}
%EndExpansion
e^{2\delta\nu_{s}}\left\vert Z_{s}\right\vert ^{2}ds\right)  ^{p/2}<\infty.
\]
From the approximating BSDE (\ref{approximating eq for general case}) we have
that, at the limit,%
\[
Y_{t}+\int_{t}^{T}\chi_{s}ds=\eta+\int_{t}^{T}F_{s}ds-\int_{t}^{T}Z_{s}%
dB_{s}.
\]
The continuity of the three integrals from the above equation imply also the
continuity of the process $Y$, but the previous convergences are not yet
sufficient to conclude that $\left(  Y,Z\right)  $ is a solution of the
considered equation. The remaining problems consist in proving that, for every
$s\in\left[  0,T\right]  $, $\mathbb{P}-a.s.~\omega\in\Omega$,%
\[
\chi_{s}=h_{s}U_{s},\quad h_{s}=H(s,Y_{s}),\quad Us\in\partial\varphi
(Y_{s})\quad\text{and}\quad F_{s}=F(s,Y_{s},Z_{s}).
\]
$\smallskip$

\noindent\textbf{Step 3.} \textit{Boundedness under the assumptions}%
\[
0\leq\ell\left(  t\right)  \equiv\ell<\sqrt{a}\quad\text{\textit{and}}%
\quad\left\vert \eta\right\vert ^{2}+\left\vert \varphi\left(  \eta\right)
\right\vert \leq c,\quad\mathbb{P}-a.s.~\omega\in\Omega.
\]

\noindent From inequality (\ref{ineq4}), written for $s=T$ it follows,
$\mathbb{P}-a.s.$,%
\[
\left\vert Y_{t}^{\varepsilon}\right\vert ^{2}+\varphi_{\varepsilon}\left(
Y_{t}^{\varepsilon}\right)  \leq C\left(  c+%
%TCIMACRO{\dint _{0}^{T}}%
%BeginExpansion
{\displaystyle\int_{0}^{T}}
%EndExpansion
\rho\left(  r\right)  dr\right)  =C^{\prime},\quad\text{for all }t\in\left[
0,T\right]  .
\]

\hfill
\end{proof}

Starting with this point, the proofs of Theorem \ref{Th. for strong existence}
and Theorem \ref{Th. for weak existence} will take two separate paths.

\section{Strong existence and uniqueness for $H\left(  t,y\right)  \equiv
H_{t}$}

We will continue in this section the proof of Theorem
\ref{Th. for strong existence}.\medskip

\begin{proof}
We continue the proof of the existence of a solution. Under the assumptions of
Step 3 (Section 3) we prove that $\left\{  Y^{\varepsilon}:0<\varepsilon
\leq1\right\}  $ is a Cauchy sequence. To simplify the presentation of this
task we assume $k=1$.

The form of the matrix $H$ leads to%
\[
H\nabla\varphi_{\varepsilon_{n}}\left(  Y^{\varepsilon_{n}}\right)
\rightharpoonup HU,\quad\text{weakly in }L^{2}(\Omega\times\left(  0,T\right)
;\mathbb{R}^{d}),
\]
that is%
\[
\lim_{n\rightarrow\infty}\mathbb{E}\int_{0}^{T}H_{r}\nabla\varphi
_{\varepsilon_{n}}\left(  Y_{r}^{\varepsilon_{n}}\right)  dr=\mathbb{E}%
\int_{0}^{T}H_{r}U_{r}dr.
\]
Starting from here, by the symmetric and strictly positive matrix $H_{s}%
{}^{-1}$ we will understand the inverse of the matrix $H_{s}$ and not the
inverse of the function $H$.

We have%
\[
H_{t}{}^{-1/2}=H_{T}{}^{-1/2}+%
%TCIMACRO{\dint _{t}^{T}}%
%BeginExpansion
{\displaystyle\int_{t}^{T}}
%EndExpansion
D_{s}ds\quad\text{and}\quad H_{t}^{-1}=H_{T}^{-1}+%
%TCIMACRO{\dint _{t}^{T}}%
%BeginExpansion
{\displaystyle\int_{t}^{T}}
%EndExpansion
\tilde{D}_{s}ds,
\]
where $D_{s}=-\dfrac{1}{2}H_{s}^{-3/2}\dfrac{d}{ds}H_{s}$ and $\tilde{D}%
_{s}=-\dfrac{d}{ds}H_{s}^{-1}$ are $\mathbb{R}^{d\times d}-$valued
progressively measurable stochastic processes such that, $\mathbb{P-}a.s.$,
$|D_{s}|\leq C=\frac{1}{2}b^{3/2}\Lambda$ and $|\tilde{D}_{s}|\leq\Lambda$.

Denote $\Delta_{s}^{\varepsilon,\delta}=H_{s}^{-1/2}\left(  Y_{s}%
^{\varepsilon}-Y_{s}^{\delta}\right)  $. We have%
\[%
\begin{array}
[c]{l}%
\Delta_{t}^{\varepsilon,\delta}=-%
%TCIMACRO{\dint _{t}^{T}}%
%BeginExpansion
{\displaystyle\int_{t}^{T}}
%EndExpansion
dH_{s}^{-1/2}\left(  Y_{s}^{\varepsilon}-Y_{s}^{\delta}\right)  -%
%TCIMACRO{\dint _{t}^{T}}%
%BeginExpansion
{\displaystyle\int_{t}^{T}}
%EndExpansion
H_{s}^{-1/2}d(Y_{s}^{\varepsilon}-Y_{s}^{\delta})\medskip\\
\quad=%
%TCIMACRO{\dint _{t}^{T}}%
%BeginExpansion
{\displaystyle\int_{t}^{T}}
%EndExpansion
d\mathcal{K}_{s}^{\varepsilon,\delta}-%
%TCIMACRO{\dint _{t}^{T}}%
%BeginExpansion
{\displaystyle\int_{t}^{T}}
%EndExpansion
\mathcal{Z}_{s}^{\varepsilon,\delta}dB_{s},
\end{array}
\]
where%
\[%
\begin{array}
[c]{l}%
d\mathcal{K}_{s}^{\varepsilon,\delta}=D_{s}\left(  Y_{s}^{\varepsilon}%
-Y_{s}^{\delta}\right)  ds+H_{s}^{-1/2}\left[  F\left(  s,Y_{s}^{\varepsilon
},Z_{s}^{\varepsilon}\right)  -F\left(  s,Y_{s}^{\delta},Z_{s}^{\delta
}\right)  \right]  ds\medskip\\
\quad\quad\quad-H_{s}^{-1/2}\left[  H_{s}\nabla\varphi_{\varepsilon}\left(
Y_{s}^{\varepsilon}\right)  -H_{s}\nabla\varphi_{\delta}(Y_{s}^{\delta
})\right]  ds\medskip\\
\quad\quad\quad=D_{s}\left(  Y_{s}^{\varepsilon}-Y_{s}^{\delta}\right)
ds+H_{s}^{-1/2}\left[  F\left(  s,Y_{s}^{\varepsilon},Z_{s}^{\varepsilon
}\right)  -F\left(  s,Y_{s}^{\delta},Z_{s}^{\delta}\right)  \right]
ds\medskip\\
\quad\quad\quad-H_{s}^{1/2}\left[  \nabla\varphi_{\varepsilon}\left(
Y_{s}^{\varepsilon}\right)  -\nabla\varphi_{\delta}(Y_{s}^{\delta})\right]  ds
\end{array}
\]
and $\mathcal{Z}_{s}^{\varepsilon,\delta}=H_{s}^{-1/2}\left(  Z_{s}%
^{\varepsilon}-Z_{s}^{\delta}\right)  $. By denoting with $C$ a generic
positive constant independent of $\varepsilon$ and $\delta$ that can change
from one line to another we obtain that%
\[%
\begin{array}
[c]{l}%
\left\langle \Delta_{s}^{\varepsilon,\delta},d\mathcal{K}_{s}^{\varepsilon
,\delta}\right\rangle \leq C\left(  |D_{s}|+L\left(  s\right)  \right)
|Y_{s}^{\varepsilon}-Y_{s}^{\delta}|^{2}ds+Cl|Y_{s}^{\varepsilon}%
-Y_{s}^{\delta}||Z_{s}^{\varepsilon}-Z_{s}^{\delta}|\medskip\\
\quad\quad\quad\quad\quad-\left\langle \nabla\varphi_{\varepsilon}\left(
Y_{s}^{\varepsilon}\right)  -\nabla\varphi_{\delta}(Y_{s}^{\delta}%
),Y_{s}^{\varepsilon}-Y_{s}^{\delta}\right\rangle \medskip\\
\quad\quad\quad\quad\leq C\left(  |D_{s}|+L\left(  s\right)  \right)
|Y_{s}^{\varepsilon}-Y_{s}^{\delta}|^{2}ds+Cl|Y_{s}^{\varepsilon}%
-Y_{s}^{\delta}||\mathcal{Z}_{s}^{\varepsilon,\delta}|ds\medskip\\
\quad\quad\quad\quad\quad+\left(  \varepsilon+\delta\right)  \left\vert
\nabla\varphi_{\varepsilon}\left(  Y_{s}^{\varepsilon}\right)  \right\vert
|\nabla\varphi_{\delta}(Y_{s}^{\delta})|ds.
\end{array}
\]
Therefore, from the formula of $\Delta_{s}^{\varepsilon,\delta}$ we have%
\[
\left\langle \Delta_{s}^{\varepsilon,\delta},d\mathcal{K}_{s}^{\varepsilon
,\delta}\right\rangle \leq\left(  \varepsilon+\delta\right)  \left\vert
\nabla\varphi_{\varepsilon}\left(  Y_{s}^{\varepsilon}\right)  \right\vert
|\nabla\varphi_{\delta}(Y_{s}^{\delta})|ds+|\Delta_{s}^{\varepsilon,\delta
}|^{2}dV_{s}+\frac{1}{4}|\mathcal{Z}_{s}^{\varepsilon,\delta}|^{2},
\]
where, for $\tilde{C}=\tilde{C}(l,a,b,\Lambda)>0$, $V_{t}=\tilde{C}\int
_{0}^{t}(|D_{s}|+L(s))ds$. We apply now Proposition \ref{ineq cond exp} (see
Annex 6.3) with $p\geq2,$ $\lambda=1/2,$ $D=N\equiv0$ and we obtain, for a
positive constant $C=C(l,a,b,p)$ and for $C_{1}>0$ well chosen,

$C_{1}\mathbb{E}\sup_{s\in\left[  0,T\right]  }|Y_{s}^{\varepsilon}%
-Y_{s}^{\delta}|^{p}+\mathbb{E}%
%TCIMACRO{\dint _{0}^{T}}%
%BeginExpansion
{\displaystyle\int_{0}^{T}}
%EndExpansion
|Z_{s}^{\varepsilon}-Z_{s}^{\delta}|^{2}ds\smallskip$

$\leq\mathbb{E}\sup_{s\in\left[  0,T\right]  }e^{pV_{s}}|\Delta_{s}%
^{\varepsilon,\delta}|^{p}+\mathbb{E}\left(
%TCIMACRO{\dint _{0}^{T}}%
%BeginExpansion
{\displaystyle\int_{0}^{T}}
%EndExpansion
e^{2V_{s}}|\mathcal{Z}_{s}^{\varepsilon,\delta}|^{2}ds\right)  ^{p/2}%
\smallskip$

$\leq C(\varepsilon+\delta)\mathbb{E}\left(
%TCIMACRO{\dint _{0}^{T}}%
%BeginExpansion
{\displaystyle\int_{0}^{T}}
%EndExpansion
e^{2V_{s}}\left\vert \nabla\varphi_{\varepsilon}\left(  Y_{s}^{\varepsilon
}\right)  \right\vert |\nabla\varphi_{\delta}(Y_{s}^{\delta})|ds\right)
^{p/2}\smallskip$

$\leq C(\varepsilon+\delta)\left(  \mathbb{E}\left(
%TCIMACRO{\dint _{0}^{T}}%
%BeginExpansion
{\displaystyle\int_{0}^{T}}
%EndExpansion
|\nabla\varphi_{\varepsilon}\left(  Y_{s}^{\varepsilon}\right)  |^{2}%
ds\right)  ^{p/2}+\mathbb{E}\left(
%TCIMACRO{\dint _{0}^{T}}%
%BeginExpansion
{\displaystyle\int_{0}^{T}}
%EndExpansion
|\nabla\varphi_{\delta}\left(  Y_{s}^{\delta}\right)  |^{2}ds\right)
^{p/2}\right)  ,\smallskip$

\noindent which implies, according to (\ref{ineq5}), that $\left\{
Y^{\varepsilon}:0<\varepsilon\leq1\right\}  $ is a Cauchy sequence.

With standard arguments, passing to the limit in the approximating equation
(\ref{approximating eq for general case}) we infer that%
\[
Y_{t}+\int_{t}^{T}H_{s}U_{s}ds=\eta+\int_{t}^{T}F(s,Y_{s},Z_{s})ds-\int
_{t}^{T}Z_{s}dB_{s},\quad\forall t\in\left[  0,T\right]  .
\]
From (\ref{space of sol for approximating equation}), by Fatou's Lemma,
(\ref{boundedness with weight}) easily follows. Moreover, since $\nabla
\varphi_{\varepsilon}(x)\in\partial\varphi(J_{\varepsilon}x)$ we have, on the
subsequence $\varepsilon_{n}$,%
\[
\mathbb{E}%
%TCIMACRO{\dint _{0}^{T}}%
%BeginExpansion
{\displaystyle\int_{0}^{T}}
%EndExpansion
\left\langle \nabla\varphi_{\varepsilon_{n}}(Y_{t}^{\varepsilon_{n}}%
),v_{t}-Y_{t}^{\varepsilon_{n}}\right\rangle dt+\mathbb{E}%
%TCIMACRO{\dint _{0}^{T}}%
%BeginExpansion
{\displaystyle\int_{0}^{T}}
%EndExpansion
\varphi(J_{\varepsilon_{n}}(Y_{t}^{\varepsilon_{n}}))dt\leq\mathbb{E}%
%TCIMACRO{\dint _{0}^{T}}%
%BeginExpansion
{\displaystyle\int_{0}^{T}}
%EndExpansion
\varphi(v_{t})dt,
\]
for every progressively measurable continuous stochastic process $v$. Hence
$U_{s}\in\partial\varphi(Y_{s})$ for every $s\in\left[  0,T\right]  ,$
$\mathbb{P}-a.s.~\omega\in\Omega$ and we can conclude that the triplet
$\left(  Y,Z,K\right)  $ is a strong solution of the $BSVI\left(  H\left(
t\right)  ,\varphi,F\right)  $.\medskip

\noindent\textbf{Uniqueness. }Suppose that the $BSVI\left(  H\left(  t\right)
,\varphi,F\right)  $ admits two strong solutions, denoted by $\left(
Y,Z,K\right)  $ and respectively $(\tilde{Y},\tilde{Z},\tilde{K})$, with the
processes $K$ and $\tilde{K}$ represented as%
\[
K_{t}=\int_{0}^{t}U_{s}ds\quad\text{and}\quad\tilde{K}_{t}=\int_{0}^{t}%
\tilde{U}_{s}ds.
\]
Following the same arguments found in the existence part of the theorem,
denoting $\Delta_{s}=H_{s}^{-1/2}(Y_{s}-\tilde{Y}_{s})$, we have%
\[
\Delta_{t}=%
%TCIMACRO{\dint _{t}^{T}}%
%BeginExpansion
{\displaystyle\int_{t}^{T}}
%EndExpansion
d\mathcal{K}_{s}-%
%TCIMACRO{\dint _{t}^{T}}%
%BeginExpansion
{\displaystyle\int_{t}^{T}}
%EndExpansion
\mathcal{Z}_{s}dB_{s},
\]
where%
\[
d\mathcal{K}_{s}=D_{s}(Y_{s}-\tilde{Y}_{s})ds+H_{s}^{-1/2}[F\left(
s,Y_{s},Z_{s}\right)  -F(s,\tilde{Y}_{s},\tilde{Z}_{s})]ds-H_{s}^{1/2}%
(U_{s}-\tilde{U}_{s})ds
\]
and $\mathcal{Z}_{s}=H_{s}^{-1/2}(Z_{s}-\tilde{Z}_{s})$.

Since $Y$ and $\tilde{Y}$ are two solutions of the equation, $U_{s}\in
\partial\varphi(Y_{s})$ and $\tilde{U}_{s}\in\partial\varphi(\tilde{Y}_{s})$,
$\forall s\in\left[  0,T\right]  $, $\mathbb{P}-a.s.~\omega\in\Omega$,%
\[
\left\langle Y_{s}-\tilde{Y}_{s},U_{s}-\tilde{U}_{s}\right\rangle \geq0
\]
and we obtain, for a positive constant $\bar{C}=\bar{C}(l,a,b)$,%
\[%
\begin{array}
[c]{l}%
\left\langle \Delta_{s},d\mathcal{K}_{s}\right\rangle \leq C\left(
|D_{s}|+L\left(  s\right)  \right)  |Y_{s}-\tilde{Y}_{s}|^{2}ds+Cl|Y_{s}%
-\tilde{Y}_{s}||\mathcal{Z}_{s}|ds\bigskip\\
\quad\quad\quad\quad~\leq\bar{C}|\Delta_{s}|^{2}(|D_{s}|+L(s))ds+\dfrac{1}%
{4}|\mathcal{Z}_{s}|^{2}.
\end{array}
\]
Since%
\[
\mathbb{E}\sup_{t\in\left[  0,T\right]  }(e^{pV_{t}}|\Delta_{t}|^{p})\leq
C\mathbb{E}\sup_{t\in\left[  0,T\right]  }|Y_{t}-\tilde{Y}_{t}|^{p}<+\infty
\]
we obtain by Proposition \ref{ineq cond exp} (see Annex 6.3) that%
\[
e^{pV_{t}}|\Delta_{t}|^{p}\leq\mathbb{E}^{\mathcal{F}_{t}}e^{pV_{T}}%
|\Delta_{T}|^{p}=0
\]
and the uniqueness of a strong solution for $BSVI\left(  H\left(  t\right)
,\varphi,F\right)  $ easily follows.\hfill
\end{proof}

\begin{remark}
Inequality (\ref{ineq2}) permits us to derive now some more estimations
regarding the limit processes. We write (\ref{ineq2}) for $s=T$ and, since
$\varphi(J_{\varepsilon}\left(  x\right)  )\leq\varphi_{\varepsilon}%
(x)\leq\varphi(x)$, by passing to $\liminf_{\varepsilon\rightarrow0}$ in
(\ref{ineq2}), we have for all $t\in\left[  0,T\right]  $, $\mathbb{P}%
-a.s.~\omega\in\Omega$,%
\begin{equation}%
\begin{array}
[c]{l}%
|Y_{t}|^{2}+\varphi(Y_{t})+a%
%TCIMACRO{\dint _{t}^{T}}%
%BeginExpansion
{\displaystyle\int_{t}^{T}}
%EndExpansion
|U_{r}|^{2}dr+%
%TCIMACRO{\dint _{t}^{T}}%
%BeginExpansion
{\displaystyle\int_{t}^{T}}
%EndExpansion
|Z_{r}|^{2}dr\leq|\eta|^{2}+\varphi(\eta)+2%
%TCIMACRO{\dint _{t}^{T}}%
%BeginExpansion
{\displaystyle\int_{t}^{T}}
%EndExpansion
\left\langle Y_{r},F(r,Y_{r},Z_{r})\right\rangle dr\medskip\\
\quad\quad\quad\quad\quad\quad\quad\quad\quad\quad\quad+%
%TCIMACRO{\dint _{t}^{T}}%
%BeginExpansion
{\displaystyle\int_{t}^{T}}
%EndExpansion
\left\langle U_{r},F(r,Y_{r},Z_{r})-2H_{r}Y_{r}\right\rangle dr-%
%TCIMACRO{\dint _{t}^{T}}%
%BeginExpansion
{\displaystyle\int_{t}^{T}}
%EndExpansion
\left\langle 2Y_{r}+U_{r},Z_{r}dB_{r}\right\rangle .
\end{array}
\label{estimation from remark}%
\end{equation}

\end{remark}

\section{Weak existence for $H\left(  t,y\right)  $}

We will continue in this section the proof of Theorem
\ref{Th. for weak existence}. All the apriori estimates obtained in Section 3
remain valid. In Section 4 we proved that the approximating sequence given by
BSDE (\ref{approximating eq for general case}) is a Cauchy sequence when the
matrix $H$ does not depend on the state of the system and, as a consequence,
we derived the existence and uniqueness of a strong solution for $BSVI\left(
H\left(  t\right)  ,\varphi,F\right)  $. In the current setup, allowing the
dependence on $Y$ we will situate ourselves in a Markovian framework and we
will use tightness criteria in order to prove the existence of a weak solution
for $BSVI\left(  H\left(  t,y\right)  ,\varphi,F\right)  $.

First let $b:\left[  0,T\right]  \times\mathbb{R}^{k}\rightarrow\mathbb{R}%
^{k}$, $\sigma:\left[  0,T\right]  \times\mathbb{R}^{k}\rightarrow
\mathbb{R}^{k\times k}$ be two continuous functions satisfying the classical
Lipschitz conditions, which imply the existence of a non-exploding solution
for the following SDE%
\begin{equation}
X_{s}^{t,x}=x+%
%TCIMACRO{\dint _{t}^{s}}%
%BeginExpansion
{\displaystyle\int_{t}^{s}}
%EndExpansion
b(r,X_{r}^{t,x})dr+%
%TCIMACRO{\dint _{t}^{s}}%
%BeginExpansion
{\displaystyle\int_{t}^{s}}
%EndExpansion
\sigma(r,X_{r}^{t,x})dB_{r},\quad t\leq s\leq T.
\label{SDE Markovian framework}%
\end{equation}
According to Friedmann \cite{Friedman:75} it follows that, for every
$(t,x)\in\left[  0,T\right]  \times\mathbb{R}^{k}$, the equation
(\ref{SDE Markovian framework}) admits a unique solution $X^{t,x}$. Moreover,
for $p\geq1$, there exists a positive constant $C_{p,T}$ such that%
\begin{equation}
\left\{
\begin{array}
[c]{l}%
\mathbb{E}\sup\nolimits_{s\in\left[  0,T\right]  }|X_{s}^{t,x}|^{p}\leq
C_{p,T}(1+|x|^{p})\quad\text{and}\medskip\\
\mathbb{E}\sup\nolimits_{s\in\left[  0,T\right]  }|X_{s}^{t,x}-X_{s}%
^{t^{\prime},x^{\prime}}|^{p}\leq C_{p,T}(1+|x|^{p})(|t-t^{\prime}%
|^{p/2}+|x-x^{\prime}|^{p}),
\end{array}
\right.  \label{estimations sol forward eq}%
\end{equation}
for all $x,x^{\prime}\in\mathbb{R}^{k}$ and $t,t^{\prime}\in\left[
0,T\right]  $.

Let now consider the continuous generator function $F:\left[  0,T\right]
\times\mathbb{R}^{k}\times\mathbb{R}^{d}\rightarrow\mathbb{R}^{d}$ and assume
there exist $L\in L^{2}\left(  0,T;\mathbb{R}_{+}\right)  $ such that, for all
$t\in\left[  0,T\right]  $ and $x\in\mathbb{R}^{k}$,%
\begin{equation}
\left\vert F(t,x,y^{\prime})-F(t,x,y)\right\vert \leq L\left(  t\right)
|y^{\prime}-y|\text{,}\quad\text{\textit{for all} }y,y^{\prime}\in
\mathbb{R}^{d}, \tag{$H_4^\prime$}\label{hypothesis on F - weak sol}%
\end{equation}
Given a continuous function $g:\mathbb{R}^{k}\rightarrow\mathbb{R}^{d}$,
satisfying a sublinear growth condition, consider now the $BSVI\left(
H\left(  t,y\right)  ,\varphi,F\right)  $%
\begin{equation}
\left\{
\begin{array}
[c]{l}%
Y_{s}^{t,x}+%
%TCIMACRO{\dint _{s}^{T}}%
%BeginExpansion
{\displaystyle\int_{s}^{T}}
%EndExpansion
H(r,Y_{r}^{t,x})dK_{r}^{t,x}=g(X_{T}^{t,x})+%
%TCIMACRO{\dint _{s}^{T}}%
%BeginExpansion
{\displaystyle\int_{s}^{T}}
%EndExpansion
F(r,X_{r}^{t,x},Y_{r}^{t,x})dr-%
%TCIMACRO{\dint _{s}^{T}}%
%BeginExpansion
{\displaystyle\int_{s}^{T}}
%EndExpansion
Z_{r}^{t,x}dB_{r},\quad t\leq s\leq T,\smallskip\\
dK_{r}^{t,x}\in\partial\varphi(Y_{r}^{t,x})\left(  dr\right)  ,\quad\text{for
every }r\text{.}%
\end{array}
\right.  \label{BSVI Markovian}%
\end{equation}

\begin{remark}
The utility of studying the notion of weak solution for our problem is
justified by the non-linear Feynman-Ka\c{c} representation formula. Following
the same arguments as the one from \cite{Pardoux/Rascanu:98}, for $k=1$, it
can easily be proven that $u(t,x)=Y_{t}^{t,x}$ is a continuous function and it
represents a viscosity solution for the following semilinear parabolic PDE:%
\[
\left\{
\begin{array}
[c]{l}%
\dfrac{\partial u}{\partial t}(t,x)+\mathcal{A}_{t}u(t,x)+F(t,x,u(t,x))\in
H(t,u(t,x))\partial\varphi(u(t,x)),\smallskip\\
(t,x)\in\lbrack0,T)\times\mathbb{R}^{k}\quad\quad\text{and}\quad\quad
u(T,x)=g(x),\quad\forall x\in\mathbb{R}^{k},
\end{array}
\right.
\]
where the operator $\mathcal{A}_{t}$ is the infinitesimal generator of the
Markov process $\{X_{s}^{t,x},$ $t\leq s\leq T\}$ and it is given by%
\[
\mathcal{A}_{t}v(x)=\frac{1}{2}\mathbf{Tr}[(\sigma\sigma^{\ast})(t,x)D^{2}%
v(x)]+\left\langle b(t,x),\nabla v(x)\right\rangle .
\]
However, for the multi-dimensional case, the situation changes and the proof
of the existence and uniqueness of a viscosity solution for the above system
of parabolic variational inequalities must follow the approach from Maticiuc,
Pardoux, R\u{a}\c{s}canu and Z\u{a}linescu
\cite{Maticiuc/Pardoux/Rascanu/Zalinescu:10}.
\end{remark}

\noindent More details concerning the restriction to the case when the
generator function does not depend on $Z$ can be found in the comments from
Pardoux \cite{Pardoux:99}, Section 6, page 535. Assume also that all
hypothesis given by $(H_{2})$ still hold for the deterministic matrix
$H:\left[  0,T\right]  \times\mathbb{R}^{d}\rightarrow\mathbb{R}^{d\times d}$.
For the clarity of the presentation we will omit writing the superscript
$t,x$, especially when dealing with sequences of approximating equations and solutions.

Consider now the Skorokhod space $\mathcal{D}(\left[  0,T\right]
;\mathbb{R}^{m})$ of c\`{a}dl\`{a}g functions $y:\left[  0,T\right]
\rightarrow\mathbb{R}^{m}$ (i.e. right continuous and with left-hand side
limit). It can be shown (see Billingsley \cite{Billingsley:99}) that, although
$\mathcal{D}(\left[  0,T\right]  ;\mathbb{R}^{m})$ is not a complete space
with respect to the Skorokhod metric, there exists a topologically equivalent
metric with respect to which it is complete and that the Skorokhod space is a
Polish space. The space of continuous functions $C(\left[  0,T\right]
;\mathbb{R}^{m})$, equipped with the supremum norm topology is a subspace of
$\mathcal{D}(\left[  0,T\right]  ;\mathbb{R}^{m})$; the Skorokhod topology
restricted to $C(\left[  0,T\right]  ;\mathbb{R}^{m})$ coincides with the
uniform topology. We will use on $\mathcal{D}(\left[  0,T\right]
;\mathbb{R}^{m})$ the Meyer-Zheng topology, which is the topology of
convergence in measure on $\left[  0,T\right]  $, weaker than the Skorokhod
topology. The Borel $\sigma-$field for the Meyer-Zheng topology is the
canonical $\sigma-$field as for Skorokhod topology. Note that for the
Meyer-Zheng topology, $\mathcal{D}(\left[  0,T\right]  ;\mathbb{R}^{m})$ is a
metric space but not a Polish space. Contrary to the Skorokhod topology, the
Meyer-Zheng topology on the product space is the product topology.$\smallskip$

\noindent We continue now the proof of Theorem \ref{Th. for weak existence}%
.\medskip

\begin{proof}
For any fixed $n\geq1$ consider the following approximating equation, which is
in fact BSDE (\ref{approximating eq for general case}) from Section 3, adapted
to our new setup. We have, $\mathbb{P}-a.s.~\omega\in\Omega$,%
\begin{equation}
Y_{t}^{n}+%
%TCIMACRO{\dint _{t}^{T}}%
%BeginExpansion
{\displaystyle\int_{t}^{T}}
%EndExpansion
H\left(  s,Y_{s}^{n}\right)  \nabla\varphi_{1/n}\left(  Y_{s}^{n}\right)
ds=g(X_{T}^{t,x})+%
%TCIMACRO{\dint _{t}^{T}}%
%BeginExpansion
{\displaystyle\int_{t}^{T}}
%EndExpansion
F\left(  s,X_{s},Y_{s}^{n}\right)  ds-%
%TCIMACRO{\dint _{t}^{T}}%
%BeginExpansion
{\displaystyle\int_{t}^{T}}
%EndExpansion
Z_{s}^{n}dB_{s},\quad\forall t\in\left[  0,T\right]  .
\label{approx eq for weak existence}%
\end{equation}
The estimations obtained in Section 3, Lemma
\ref{Lemma with the estimations from Step 1} apply also to the triplet
$(Y^{n},Z^{n},U^{n})=(Y^{n},Z^{n},\nabla\varphi_{1/n}\left(  Y^{n}\right)
),$which satisfies the uniform boundedness condition given by
(\ref{ineq Y,Z,U}) with the positive constant $C=C(a,b,\Lambda,L(\cdot))$ now
independent of $n$. We will prove a weakly convergence in the sense of the
Meyer-Zheng topology, that is the laws converge weakly if we equip the space
of paths with the topology of convergence in $dt-$measure.$\smallskip$

In the sequel we will employ the following notations:%
\[
M_{t}^{n}=%
%TCIMACRO{\dint _{0}^{t}}%
%BeginExpansion
{\displaystyle\int_{0}^{t}}
%EndExpansion
Z_{s}^{n}dB_{s}\quad\text{and}\quad K_{t}^{n}=%
%TCIMACRO{\dint _{0}^{t}}%
%BeginExpansion
{\displaystyle\int_{0}^{t}}
%EndExpansion
\nabla\varphi_{1/n}\left(  Y_{s}^{n}\right)  ds.
\]
Our goal is to prove the tightness of the sequence $\{Y^{n},M^{n}\}_{n}$ with
respect to the Meyer-Zheng topology. For doing this we must prove the uniform
boundedness (with respect to $n$) for quantities of the type%
\[
\mathrm{CV}_{T}\left(  \Psi\right)  +\mathbb{E}\sup_{s\in\lbrack0,T]}|\Psi
_{s}|,
\]
where the conditional variation $\mathrm{CV}_{T}$ is defined for any adapted
process $\Psi$ with paths a.s. in $\mathcal{D}(\left[  0,T\right]
;\mathbb{R}^{m})$ and with $\Psi_{t}$ a integrable random variable, for all
$t\in\left[  0,T\right]  $. The conditional variation of $\Psi$ is given by%
\begin{equation}
\mathrm{CV}_{T}(\Psi)\overset{def}{=}\sup_{\pi}{\sum_{i=0}^{m-1}{\mathbb{E}%
}\Big[{\big|}}\mathbb{E}^{{{\mathcal{F}}_{t_{i}}}}{[\Psi_{t_{i+1}}-\Psi
_{t_{i}}]{\big|}\Big],} \label{def cond var}%
\end{equation}
where the supremum is taken over all the partitions $\pi:t=t_{0}<t_{1}%
<\cdots<t_{m}=T$. If $\mathrm{CV}_{T}(\Psi)<\infty$ then the process $\Psi$ is
called a quasi-martingale. It is clear that if $\Psi$ is a martingale then
$\mathrm{CV}_{T}(\Psi)=0$.

We will denote by $C$ a generic constant that can vary from one line to
another, but which remains independent of $n$. Since $M^{n}$ is a
$\mathcal{F}_{t}^{B}-$martingale, we have, by using the hypothesis on $F$ and
the boundedness of $H$,

$\mathrm{CV}_{T}(Y^{n})=\sup\limits_{\pi}%
%TCIMACRO{\dsum \limits_{i=0}^{m-1}}%
%BeginExpansion
{\displaystyle\sum\limits_{i=0}^{m-1}}
%EndExpansion
{{\mathbb{E}}\Big[{\big|}}\mathbb{E}^{{{\mathcal{F}}_{t_{i}}}}{[Y_{t_{i+1}%
}^{n}-Y_{t_{i}}^{n}]{\big|}\Big]\leq}\mathbb{E}%
%TCIMACRO{\dint _{0}^{T}}%
%BeginExpansion
{\displaystyle\int_{0}^{T}}
%EndExpansion
|F(s,X_{s},Y_{s}^{n})|ds+%
%TCIMACRO{\dint _{0}^{T}}%
%BeginExpansion
{\displaystyle\int_{0}^{T}}
%EndExpansion
|H(Y_{s}^{n})|d\left\updownarrow K^{n}\right\updownarrow _{s}$\medskip

$\quad\quad\quad\quad\leq C\mathbb{E}%
%TCIMACRO{\dint _{0}^{T}}%
%BeginExpansion
{\displaystyle\int_{0}^{T}}
%EndExpansion
(1+|Y_{s}^{n}|)ds+b\left\updownarrow K^{n}\right\updownarrow _{T}.\smallskip$

\noindent Since $\left\updownarrow K^{n}\right\updownarrow _{T}=%
%TCIMACRO{\dint _{0}^{T}}%
%BeginExpansion
{\displaystyle\int_{0}^{T}}
%EndExpansion
|U_{s}^{n}|ds\leq\sqrt{T}{\Big(}%
%TCIMACRO{\dint _{0}^{T}}%
%BeginExpansion
{\displaystyle\int_{0}^{T}}
%EndExpansion
|U_{s}^{n}|^{2}ds{\Big)}^{1/2}\leq C$ it infers, along with the uniform
boundedness condition given by (\ref{ineq Y,Z,U}) that%
\[
\sup_{n\geq1}{\Big(}\mathrm{CV}_{T}(Y^{n})+\mathbb{E}\sup_{s\in\left[
0,T\right]  }|Y_{s}^{n}|{\Big)}<+\infty.
\]
\noindent For the rest of the quantities, by standard calculus and using
(\ref{ineq Y,Z,U}) we have the following estimations.$\smallskip$

$\mathrm{CV}_{T}(M^{n})=0$ because $M^{n}$ is a $\mathcal{F}_{t}-$martingale.
Using the Burkholder-Davis-Gundy inequality we obtain the second boundedness
which involves $M^{n}$.\medskip

$\mathbb{E}\sup\limits_{t\in\left[  0,T\right]  }|M_{t}^{n}|=\mathbb{E}%
\sup\limits_{t\in\left[  0,T\right]  }\left\vert
%TCIMACRO{\dint _{0}^{t}}%
%BeginExpansion
{\displaystyle\int_{0}^{t}}
%EndExpansion
Z_{s}^{n}dB_{s}\right\vert \leq3\mathbb{E}\left(
%TCIMACRO{\dint _{0}^{T}}%
%BeginExpansion
{\displaystyle\int_{0}^{T}}
%EndExpansion
|Z_{s}^{n}|^{2}ds\right)  ^{1/2}\leq3\left(  \mathbb{E}%
%TCIMACRO{\dint _{0}^{T}}%
%BeginExpansion
{\displaystyle\int_{0}^{T}}
%EndExpansion
|Z_{s}^{n}|^{2}ds\right)  ^{1/2}\leq C$.\medskip

\noindent Therefore, taking the supremum over $n\geq1$ we obtain that the
conditions from the tightness criteria in $\mathcal{D}(\left[  0,T\right]
;\mathbb{R}^{d})\times\mathcal{D}(\left[  0,T\right]  ;\mathbb{R}^{d}%
)[\equiv\mathcal{D}(\left[  0,T\right]  ;\mathbb{R}^{d+d})]$ for the sequence
$\{(Y^{n},M^{n})\}_{n}$ are verified. Using the Prohorov theorem, we have that
there exists a subsequence, still denoted with $n$, such that, as
$n\rightarrow\infty$,%
\[
(X,B,Y^{n},M^{n})\longrightarrow(X,B,Y,M),\quad\text{in law}%
\]
in $C(\left[  0,T\right]  ;\mathbb{R}^{k+k})\times\mathcal{D}(\left[
0,T\right]  ;\mathbb{R}^{d+d})$. We equipped the previous space with the
product of the topology of uniform convergence on the first factor and the
topology of convergence in measure on the second factor. For each $0\leq s\leq
t$, the mapping $\left(  x,y\right)  \rightarrow\int_{s}^{t}F(x(r),y(r))dr$ is
continuous from $C(\left[  0,T\right]  ;\mathbb{R}^{k})\times\mathcal{D}%
(\left[  0,T\right]  ;\mathbb{R}^{d})$ topologically equipped in the same
manner, into $\mathbb{R}$. By the Skorokhod theorem, we can choose now a
probability space $\left(  \bar{\Omega},\mathcal{\bar{F}},\mathbb{\bar{P}%
}\right)  $ (it is in fact $([0,1],\mathcal{B}_{[0,1]},\mu)$) on which we
define the processes%
\[
\{(\bar{X}^{n},\bar{B}^{n},\bar{Y}^{n},\bar{M}^{n})\}_{n}\quad\text{and}%
\quad(\bar{X},\bar{B},\bar{Y},\bar{M}),
\]
having the same law as $\{(X,B,Y^{n},M^{n})\}_{n}$ and $(X,B,Y,M)$,
respectively, such that, in the product space $C(\left[  0,T\right]
;\mathbb{R}^{k+k})\times\mathcal{D}(\left[  0,T\right]  ;\mathbb{R}^{d+d})$,
as $n\rightarrow\infty$,%
\[
(\bar{X}^{n},\bar{B}^{n},\bar{Y}^{n},\bar{M}^{n})\overset{\mathbb{\bar{P}%
-}a.s.}{\longrightarrow}(\bar{X},\bar{B},\bar{Y},\bar{M}).
\]
Moreover, for each $n\in\mathbb{N}^{\ast}$, $(\bar{X}^{n},\bar{Y}^{n})$
satisfy, for $t\in\left[  0,T\right]  $, $\mathbb{\bar{P}}-a.s.~\omega\in
\bar{\Omega}$,%
\begin{equation}
d\bar{X}_{s}^{n}=b(s,\bar{X}_{s}^{n})ds+\sigma(s,\bar{X}_{s}^{n})d\bar{B}%
_{s}^{n},\quad t\leq s\leq T,\quad\bar{X}_{t}^{n}=x\quad\text{and}
\label{n_bar_forward_equation}%
\end{equation}%
\begin{equation}
\bar{Y}_{t}^{n}+%
%TCIMACRO{\dint _{t}^{T}}%
%BeginExpansion
{\displaystyle\int_{t}^{T}}
%EndExpansion
H\left(  s,\bar{Y}_{s}^{n}\right)  \nabla\varphi_{1/n}\left(  \bar{Y}_{s}%
^{n}\right)  ds=g(\bar{X}_{T}^{n})+%
%TCIMACRO{\dint _{t}^{T}}%
%BeginExpansion
{\displaystyle\int_{t}^{T}}
%EndExpansion
F\left(  s,\bar{X}_{s}^{n},\bar{Y}_{s}^{n}\right)  ds-(\bar{M}_{T}^{n}-\bar
{M}_{t}^{n}). \label{n_bar_equation}%
\end{equation}
We focus now to the issue of passing to the limit and to the identification of
a solution for our problem. Since $dK_{s}^{n}=\nabla\varphi_{1/n}(Y_{s}%
^{n})ds\in\partial\varphi(J_{n}(Y_{s}^{n}))(ds)$ we have, for all
$v\in\mathbb{R}^{d}$ and $0\leq t\leq s_{1}\leq s_{2}$,%
\[
\int_{s_{1}}^{s_{2}}\varphi(J_{n}(Y_{s}^{n}))ds\leq\int_{s_{1}}^{s_{2}}%
(J_{n}(Y_{s}^{n})-v)\nabla\varphi_{1/n}(Y_{s}^{n})ds+\int_{s_{1}}^{s_{2}%
}\varphi(v)ds.
\]
Using similar arguments to the ones found in Pardoux and R\u{a}\c{s}canu
\cite{Pardoux/Rascanu:09}, Proposition 1.19, it easily follows that, also for
all $v\in\mathbb{R}^{d},$ $0\leq t\leq s_{1}\leq s_{2}$ and every
$A\in\mathcal{\bar{F}},$%
\begin{equation}
\mathbb{\bar{E}}\int_{s_{1}}^{s_{2}}\mathbf{1}_{A}\varphi(J_{n}(\bar{Y}%
_{s}^{n}))ds\leq\mathbb{\bar{E}}\int_{s_{1}}^{s_{2}}\mathbf{1}_{A}(J_{n}%
(\bar{Y}_{s}^{n})-v)\nabla\varphi_{1/n}(\bar{Y}_{s}^{n})ds+\mathbb{\bar{E}%
}\int_{s_{1}}^{s_{2}}\mathbf{1}_{A}\varphi(v)ds, \label{n_bar_inequality}%
\end{equation}
that is, $\mathbb{\bar{P}-}a.s.~\omega\in\bar{\Omega},$ $\nabla\varphi
_{1/n}(\bar{Y}_{s}^{n})\in\partial\varphi(J_{n}(\bar{Y}_{s}^{n}))$, for all
$s\in\left[  t,T\right]  $. We write (\ref{ineq Y,Z,U}) for $\bar{Y}^{n}$ and,
by using the definition of the Yosida approximation, we obtain that there
exists a positive constant $C$, independent of $n$, such that $\mathbb{E}%
%TCIMACRO{\tint _{0}^{T}}%
%BeginExpansion
{\textstyle\int_{0}^{T}}
%EndExpansion
|Y_{s}^{n}-J_{n}(Y_{s}^{n})|^{2}ds\leq\frac{1}{n^{2}}C$. The fact that
$Y^{n}\overset{\mathcal{L}}{\sim}\bar{Y}^{n}$ yields%
\[
\mathbb{\bar{E}}%
%TCIMACRO{\dint _{0}^{T}}%
%BeginExpansion
{\displaystyle\int_{0}^{T}}
%EndExpansion
|\bar{Y}_{s}^{n}-J_{n}(\bar{Y}_{s}^{n})|^{2}ds\leq\frac{1}{n^{2}}C.
\]
Consequently, $\bar{Y}^{n}-J_{n}(\bar{Y}^{n})\longrightarrow0$ as
$n\rightarrow\infty$ in $L^{2}(\bar{\Omega}\times(0,T);\mathbb{R}^{d})$.
Therefore, $J_{n}(\bar{Y}^{n})$ converges also in $L^{2}(\bar{\Omega}%
\times(0,T);\mathbb{R}^{d})$ to $\bar{Y}$ when $n\rightarrow\infty$. The
boundedness (\ref{ineq Y,Z,U}) also implies the existence of a process
$\bar{U}$ such that%
\[
\nabla\varphi_{1/n}(\bar{Y}^{n})\rightharpoonup\bar{U}\quad\text{as
}n\rightarrow\infty\text{, in}\quad L^{2}(\bar{\Omega}\times(0,T);\mathbb{R}%
^{d}).
\]
In addition, passing to $\liminf_{n\rightarrow+\infty}$ in
(\ref{n_bar_inequality}), due to the lower-semicontinuity of $\varphi$ we
obtain, for all $v\in\mathbb{R}^{d}$ and all $0\leq t\leq s_{1}\leq s_{2}$,
$\mathbb{P}-a.s.~\omega\in\Omega$,%
\[
\int_{s_{1}}^{s_{2}}\varphi(\bar{Y}_{s})ds\leq\int_{s_{1}}^{s_{2}}(\bar{Y}%
_{s}-v)\bar{U}_{s}ds+\int_{s_{1}}^{s_{2}}\varphi(v)ds,
\]
which means $d\bar{K}_{s}\overset{def}{=}\bar{U}_{s}ds\in\partial\varphi
(\bar{Y}_{s})(ds)$.

Finally, we pass to the limit, as $n\rightarrow\infty$, in the equations
(\ref{n_bar_forward_equation}) and (\ref{n_bar_equation}). The convergence of
$(\bar{X}^{n},\bar{B}^{n},\bar{Y}^{n},\bar{M}^{n})$ to $(\bar{X},\bar{B}%
,\bar{Y},\bar{M})$ implies, $\mathbb{P}-a.s.~\omega\in\Omega$,%
\[
\bar{X}_{s}=x+%
%TCIMACRO{\dint _{t}^{s}}%
%BeginExpansion
{\displaystyle\int_{t}^{s}}
%EndExpansion
b(r,\bar{X}_{r})dr+%
%TCIMACRO{\dint _{t}^{s}}%
%BeginExpansion
{\displaystyle\int_{t}^{s}}
%EndExpansion
\sigma(r,\bar{X}_{r})d\bar{B}_{r},\quad t\leq s\leq T
\]
and%
\[
\bar{Y}_{t}+%
%TCIMACRO{\dint _{t}^{T}}%
%BeginExpansion
{\displaystyle\int_{t}^{T}}
%EndExpansion
H\left(  s,\bar{Y}_{s}\right)  \bar{U}_{s}ds=g(\bar{X}_{T})+%
%TCIMACRO{\dint _{t}^{T}}%
%BeginExpansion
{\displaystyle\int_{t}^{T}}
%EndExpansion
F\left(  s,\bar{X}_{s},\bar{Y}_{s}\right)  ds-(\bar{M}_{T}-\bar{M}_{t}).
\]
Since the processes $\bar{Y}$ and $\bar{M}$ are c\`{a}dl\`{a}g the above
equality takes place for any $t\in\left[  0,T\right]  $.

Summarizing, we obtained that the collection $(\bar{\Omega},\mathcal{\bar{F}%
},\mathbb{\bar{P}},\mathcal{F}_{t}^{\bar{Y},\bar{M}},\bar{Y}_{t},\bar{M}%
_{t},\bar{K}_{t})_{t\in\left[  0,T\right]  }$ is a weak solution of
Eq.(\ref{BSVI Markovian}), in the sense of Definition
(\ref{Def of weak solution}), and the proof is now complete.

\hfill
\end{proof}

\begin{remark}
Alternatively, one can use another approximating equation instead of
(\ref{approx eq for weak existence}) to prove the existence of a weak
solution. This new approach comes with additional benefits from the
perspective of constructing numerical approximating schemes for our stochastic
variational inequality. For $n\in\mathbb{N}^{\ast}$ we consider a partition of
the time interval $\left[  0,T\right]  $ of the form $0=t_{0}<t_{1}%
<...<t_{n}=T$ with $t_{i}=\frac{iT}{n}$ for every $i=\overline{0,n-1}$ and
define%
\begin{equation}
\left\{
\begin{array}
[c]{l}%
Y_{t_{n}}^{n}=\eta,\medskip\\
Y_{t}^{n}+%
%TCIMACRO{\dint _{t}^{t_{i+1}}}%
%BeginExpansion
{\displaystyle\int_{t}^{t_{i+1}}}
%EndExpansion
H_{s}^{n}dK_{s}^{n}=Y_{t_{i+1}}^{n}+%
%TCIMACRO{\dint _{t}^{t_{i+1}}}%
%BeginExpansion
{\displaystyle\int_{t}^{t_{i+1}}}
%EndExpansion
F(s,X_{s},Y_{s}^{n})ds-%
%TCIMACRO{\dint _{t}^{t_{i+1}}}%
%BeginExpansion
{\displaystyle\int_{t}^{t_{i+1}}}
%EndExpansion
Z_{s}^{n}dB_{s},\text{ }\forall t\in\lbrack t_{i},t_{i+1}),\medskip\\
dK_{s}^{n}=U_{s}^{n}ds\in\partial\varphi(Y_{s}^{n})(ds),
\end{array}
\right.  \label{alternatively approx eq for weak existence}%
\end{equation}
where, for $s\in\lbrack\frac{iT}{n},\frac{(i+1)T}{n})$,%
\[
H_{s}^{n}\overset{def}{=}\frac{n}{T}%
%TCIMACRO{\dint _{s-\frac{T}{n}}^{s}}%
%BeginExpansion
{\displaystyle\int_{s-\frac{T}{n}}^{s}}
%EndExpansion
\mathbb{E}^{\mathcal{F}_{r}}\left(  H\left(  r,Y_{r+\frac{2T}{n}}^{n}\right)
\right)  dr.
\]
For the consistence of (\ref{alternatively approx eq for weak existence}) we
must extend $Y_{t}^{n}=\eta$, $U_{t}^{n}=0$ for $t\notin\left[  0,T\right]  $
and, $\mathbb{P}-a.s.~\omega\in\Omega$, $U_{t}^{n}\in\partial\varphi(Y_{t}%
^{n})$ a.e. $t\in(0,T)$. The application $s\rightarrow H_{s}^{n}$ is a bounded
$C^{1}$ progressively measurable matrix on each interval $(t_{i},t_{i+1})$;
$H^{n}$ and its inverse $[H^{n}]^{-1}$ satisfy (\ref{hypothesis on H}). We
highlight that all the constants that appear in (\ref{hypothesis on H}) remain
independent of $n$. Also, it is clear that, for any continuous process $V$,%
\[
\frac{n}{T}%
%TCIMACRO{\dint _{s-\frac{T}{n}}^{s}}%
%BeginExpansion
{\displaystyle\int_{s-\frac{T}{n}}^{s}}
%EndExpansion
\mathbb{E}^{\mathcal{F}_{r}}\left(  H\left(  r,V_{r+\frac{2T}{n}}\right)
\right)  dr\underset{n\rightarrow\infty}{\longrightarrow}H(s,V_{s}).
\]
By Theorem \ref{Th. for strong existence} the triplet $(Y^{n},Z^{n},U^{n})$ is
uniquely defined by Eq.(\ref{alternatively approx eq for weak existence}) as
its strong solution. One can rewrite Eq.
(\ref{alternatively approx eq for weak existence}) under a global form on the
entire time interval $\left[  0,T\right]  $. We have, $\mathbb{P}%
-a.s.~\omega\in\Omega$,%
\begin{equation}
\left\{
\begin{array}
[c]{l}%
Y_{t_{n}}^{n}=\eta,\medskip\\
Y_{t}^{n}+%
%TCIMACRO{\dint _{t}^{T}}%
%BeginExpansion
{\displaystyle\int_{t}^{T}}
%EndExpansion
H_{s}^{n}dK_{s}^{n}=\eta+%
%TCIMACRO{\dint _{t}^{T}}%
%BeginExpansion
{\displaystyle\int_{t}^{T}}
%EndExpansion
F(s,X_{s},Y_{s}^{n})ds-%
%TCIMACRO{\dint _{t}^{T}}%
%BeginExpansion
{\displaystyle\int_{t}^{T}}
%EndExpansion
Z_{s}^{n}dB_{s},\text{ }\forall t\in\left[  0,T\right]  ,\medskip\\
dK_{s}^{n}=U_{s}^{n}ds\in\partial\varphi(Y_{s}^{n})(ds)
\end{array}
\right.  \label{alternatively global approx eq for weak existence}%
\end{equation}
and we obtain that the triplet $(Y^{n},Z^{n},U^{n})$ satisfies a boundedness
property similar to (\ref{ineq Y,Z,U}). This permits us to prove, in the same
manner as in Theorem \ref{Th. for weak existence}, the tightness criteria
followed by the existence of a weak solution.
\end{remark}

\section{Annex}

For the clarity of the proofs from the main body of this article we will group
in this section some useful results that are used throughout this paper. For
more details the interested reader can consult the monograph of Pardoux and
R\u{a}\c{s}canu \cite{Pardoux/Rascanu:09}.

\subsection{BSDEs with Lipschitz coefficient}

We first introduce the spaces that will appear in the next results. Denote by
$S_{d}^{p}\left[  0,T\right]  $, $p\geq0$, the space of progressively
measurable continuous stochastic processes $X:\Omega\times\left[  0,T\right]
\rightarrow\mathbb{R}^{d}$, such that%
\[
\left\Vert X\right\Vert _{S_{d}^{p}}=\left\{
\begin{array}
[c]{ll}%
\left(  \mathbb{E}\left\Vert X\right\Vert _{T}^{p}\right)  ^{\frac{1}{p}%
\wedge1}<{\infty}, & \;\text{if }p>0,\medskip\\
\mathbb{E}\left[  1\wedge\left\Vert X\right\Vert _{T}\right]  , & \;\text{if
}p=0,
\end{array}
\right.
\]
where $\left\Vert X\right\Vert _{T}=\sup_{t\in\left[  0,T\right]  }\left\vert
X_{t}\right\vert $. The space $(S_{d}^{p}\left[  0,T\right]  ,\left\Vert
\cdot\right\Vert _{S_{d}^{p}}),\ p\!\geq1,$ is a Banach space and $S_{d}%
^{p}\left[  0,T\right]  $, $0\leq p<1$, is a complete metric space with the
metric $\rho(Z_{1},Z_{2})=\left\Vert Z_{1}-Z_{2}\right\Vert _{S_{d}^{p}}$
(when $p=0$ the metric convergence coincides with the probability convergence).

Denote by $\Lambda_{d\times k}^{p}\left(  0,T\right)  ,\ p\in\lbrack0,{\infty
})$, the space of progressively measurable stochastic processes $Z:{\Omega
}\times(0,T)\rightarrow\mathbb{R}^{d\times k}$ such that
\[
\left\Vert Z\right\Vert _{\Lambda^{p}}=\left\{
\begin{array}
[c]{ll}%
\left[  \mathbb{E}\left(  \displaystyle\int_{0}^{T}\Vert Z_{s}\Vert
^{2}ds\right)  ^{\frac{p}{2}}\right]  ^{\frac{1}{p}\wedge1}, & \;\text{if
}p>0,\bigskip\\
\mathbb{E}\left[  1\wedge\left(  \displaystyle\int_{0}^{T}\Vert Z_{s}\Vert
^{2}ds\right)  ^{\frac{1}{2}}\right]  , & \;\text{if }p=0.
\end{array}
\right.
\]
The space $(\Lambda_{d\times k}^{p}\left(  0,T\right)  ,\left\Vert
\cdot\right\Vert _{\Lambda^{p}}),\ p\geq1,$ is a Banach space and
$\Lambda_{d\times k}^{p}\left(  0,T\right)  $, $0\leq p<1,$ is a complete
metric space with the metric $\rho(Z_{1},Z_{2})=\left\Vert Z_{1}%
-Z_{2}\right\Vert _{\Lambda^{p}}$.\medskip

Let consider the following generalized BSDE%
\begin{equation}
Y_{t}=\eta+\int_{t}^{T}\Phi\left(  s,Y_{s},Z_{s}\right)  dQ_{s}-\int_{t}%
^{T}Z_{s}dB_{s},~,\;t\in\left[  0,T\right]  ,\text{ }\mathbb{P}-a.s.~\omega
\in\Omega, \label{gen BSDE Lip}%
\end{equation}
where

\begin{itemize}
\item $\quad\eta:\Omega\rightarrow\mathbb{R}^{d}$\textit{ }is a\textit{
}$\mathcal{F}_{T}-$measurable random vector;

\item $\quad Q$ is a\textit{ }progressively measurable increasing continuous
stochastic process such that $Q_{0}=0$;

\item $\quad\Phi:\Omega\times\left[  0,T\right]  \times\mathbb{R}%
^{d}\rightarrow\mathbb{R}^{d}$\textit{ }for which we denote\textit{ }%
$\Phi_{\rho}^{\#}\left(  t\right)  \overset{def}{=}\sup_{\left\vert
y\right\vert \leq\rho}\left\vert \Phi(t,y,0)\right\vert .$
\end{itemize}

\noindent We shall assume that:\medskip

\noindent\textbf{(BSDE-LH) }:

\begin{itemize}
\item[$\left(  i\right)  $] \text{for all }$y\in\mathbb{R}^{d}$ and
$z\in\mathbb{R}^{d\times k}$ the function $\Phi\left(  \cdot,\cdot,y,z\right)
:\Omega\times\left[  0,T\right]  \rightarrow\mathbb{R}^{d}$ is progressively measurable;

\item[$\left(  ii\right)  $] there exist the progressively measurable
stochastic processes\textit{ }$L,\ell,\alpha:\Omega\times\left[  0,T\right]
\rightarrow\mathbb{R}_{+}$ \textit{\ }such that%
\[
\alpha_{t}dQ_{t}=dt\quad\quad\quad\text{and}\quad\quad\quad%
%TCIMACRO{\dint _{0}^{T}}%
%BeginExpansion
{\displaystyle\int_{0}^{T}}
%EndExpansion
\left(  L_{t}dQ_{t}+\ell_{t}^{2}dt\right)  <\infty,\;\mathbb{P}-a.s.~\omega
\in\Omega
\]
and, for all $t\in\left[  0,T\right]  $, $y,y^{\prime}\in\mathbb{R}^{d}$
and$\;z,z^{\prime}\in\mathbb{R}^{d\times k},\;\mathbb{P}-a.s.~\omega\in\Omega$%
\begin{equation}%
\begin{array}
[c]{r}%
\text{\textit{Lipschitz conditions :}}\\
\\
\\
\text{\textit{Boundedness condition :}}%
\end{array}%
\begin{array}
[c]{rl}%
\left(  a\right)  \quad & \left\vert \Phi(t,y^{\prime},z)-\Phi
(t,y,z)\right\vert \leq L_{t}|y^{\prime}-y|,\medskip\\
\left(  b\right)  \quad & |\Phi(t,y,z^{\prime})-\Phi(t,y,z)|\leq\alpha_{t}%
\ell_{t}|z^{\prime}-z|,\medskip\\
\left(  c\right)  \quad &
%TCIMACRO{\dint _{0}^{T}}%
%BeginExpansion
{\displaystyle\int_{0}^{T}}
%EndExpansion
\Phi_{\rho}^{\#}\left(  t\right)  dQ_{t}<\infty,\quad\forall\rho\geq0.
\end{array}
\label{ch5-gl1}%
\end{equation}

\end{itemize}

\noindent Remark that condition $\alpha_{t}dQ_{t}=dt$ implies%
\[
\Phi\left(  t,Y_{t},Z_{t}\right)  dQ_{t}=F\left(  t,Y_{t},Z_{t}\right)
dt+G\left(  t,Y_{t}\right)  dA_{t},
\]
where $G$ does not depend on the $z$ variable.

\noindent Let $p>1$ and $n_{p}\overset{def}{=}1\wedge\left(  p-1\right)  $.
The following existence and uniqueness result takes place.

\begin{theorem}
[See Theorem 5.29 from Pardoux and R\u{a}\c{s}canu \cite{Pardoux/Rascanu:09}%
]\label{ch5-t2aa}Suppose that the assumptions \textbf{(BSDE-LH)} are
satisfied. Consider%
\[
V_{t}=%
%TCIMACRO{\dint _{0}^{t}}%
%BeginExpansion
{\displaystyle\int_{0}^{t}}
%EndExpansion
L_{s}dQ_{s}+\dfrac{1}{n_{p}}%
%TCIMACRO{\dint _{0}^{t}}%
%BeginExpansion
{\displaystyle\int_{0}^{t}}
%EndExpansion
\ell_{s}^{2}ds.
\]
If, for all $\delta>1$,%
\begin{equation}
\mathbb{E}|e^{\delta V_{T}}\eta|^{p}+\mathbb{E}\Big(%
%TCIMACRO{\dint _{0}^{T}}%
%BeginExpansion
{\displaystyle\int_{0}^{T}}
%EndExpansion
e^{\delta V_{t}}\left\vert \Phi\left(  t,0,0\right)  \right\vert
dQ_{t}\Big)^{p}<\infty\label{ip-V-delta}%
\end{equation}
then the BSDE (\ref{gen BSDE Lip}) admits a unique solution $\left(
Y,Z\right)  \in S_{d}^{0}\left[  0,T\right]  \times\Lambda_{d\times k}%
^{0}\left(  0,T\right)  $ such that%
\[
\mathbb{E}\sup\limits_{s\in\left[  0,T\right]  }e^{\delta pV_{s}}\left\vert
Y_{s}\right\vert ^{p}+\mathbb{E}\left(
%TCIMACRO{\dint _{0}^{T}}%
%BeginExpansion
{\displaystyle\int_{0}^{T}}
%EndExpansion
e^{2\delta V_{s}}\left\vert Y_{s}\right\vert ^{2}L_{s}dQ_{s}\right)
^{p/2}+\mathbb{E}\left(
%TCIMACRO{\dint _{0}^{T}}%
%BeginExpansion
{\displaystyle\int_{0}^{T}}
%EndExpansion
e^{2\delta V_{s}}\left\vert Z_{s}\right\vert ^{2}ds\right)  ^{p/2}<\infty.
\]

\end{theorem}

Consider now the BSDE%
\begin{equation}
Y_{t}=\eta+\int_{t}^{T}F\left(  s,Y_{s},Z_{s}\right)  ds-\int_{t}^{T}%
Z_{s}dB_{s},~,\;t\in\left[  0,T\right]  ,\;\mathbb{P}-a.s.~\omega\in\Omega.
\label{ch5-ll8}%
\end{equation}
where for all $y\in\mathbb{R}^{d}$, $z\in\mathbb{R}^{d\times k}$, the function
$F\left(  \cdot,y,z\right)  :\left[  0,T\right]  \rightarrow\mathbb{R}^{d}$ is
measurable and there exist some measurable deterministic functions
$L,\kappa,\rho\in L^{1}\left(  0,T;\mathbb{R}_{+}\right)  $ and $\ell\in
L^{2}\left(  0,T;\mathbb{R}_{+}\right)  $ such that, for all $y,y^{\prime}%
\in\mathbb{R}^{d}$, $z,z^{\prime}\in\mathbb{R}^{d\times k},$ $dt-a.e.,$%
\begin{equation}%
\begin{array}
[c]{l}%
\left\vert F(t,y^{\prime},z)-F(t,y,z)\right\vert \leq L\left(  t\right)
\left(  1+\left\vert y\right\vert \vee\left\vert y^{\prime}\right\vert
\right)  |y^{\prime}-y|,\medskip\\
|F(t,y,z^{\prime})-F(t,y,z)|\leq\ell\left(  t\right)  |z^{\prime}%
-z|,\medskip\\
\left\vert F(t,y,0)\right\vert \leq\rho\left(  t\right)  +\kappa\left(
t\right)  \left\vert y\right\vert .
\end{array}
\label{ch5-ll9}%
\end{equation}
Letting $\gamma\left(  t\right)  =\kappa\left(  t\right)  +\dfrac{1}{n_{p}%
}\ell^{2}\left(  t\right)  $ and $\bar{\gamma}\left(  t\right)  =%
%TCIMACRO{\dint _{0}^{t}}%
%BeginExpansion
{\displaystyle\int_{0}^{t}}
%EndExpansion
\left(  \kappa\left(  s\right)  +\dfrac{1}{n_{p}}\ell^{2}\left(  s\right)
\right)  ds$, consider the stochastic process $\beta\in S_{1}^{0}\left[
0,T\right]  $ given by%
\[
\beta_{t}=C^{\prime}\left(  1+\left(  \mathbb{E}^{\mathcal{F}_{t}}\left\vert
\eta\right\vert ^{p}\right)  ^{1/p}\right)  \geq\left(  C_{p}\right)
^{1/p}e^{-\bar{\gamma}\left(  t\right)  }\left\{  \mathbb{E}^{\mathcal{F}_{t}%
}\left[  |e^{\bar{\gamma}\left(  T\right)  }\eta|^{p}+\left(
%TCIMACRO{\dint _{t}^{T}}%
%BeginExpansion
{\displaystyle\int_{t}^{T}}
%EndExpansion
e^{\bar{\gamma}\left(  s\right)  }\rho\left(  s\right)  ds\right)
^{p}\right]  \right\}  ^{1/p},
\]
where $C^{\prime}=C^{\prime}\left(  p,\bar{\gamma}\left(  T\right)  ,%
%TCIMACRO{\dint _{0}^{T}}%
%BeginExpansion
{\displaystyle\int_{0}^{T}}
%EndExpansion
\rho\left(  s\right)  ds\right)  $.

\noindent Denote%
\[
\nu_{t}=%
%TCIMACRO{\dint _{0}^{t}}%
%BeginExpansion
{\displaystyle\int_{0}^{t}}
%EndExpansion
L\left(  s\right)  \left[  \mathbb{E}^{\mathcal{F}_{s}}\left\vert
\eta\right\vert ^{p}\right]  ^{1/p}\quad\quad\text{and}\quad\quad\theta
=\sup_{t\in\left[  0,T\right]  }\left(  \mathbb{E}^{\mathcal{F}_{t}}\left\vert
\eta\right\vert ^{p}\right)  ^{1/p}~.
\]

\begin{theorem}
\label{Corollary existence sol}Let $p>1$ and the assumptions (\ref{ch5-ll9})
be satisfied. If $\mathbb{E}e^{\delta\theta}<\infty,$ for all $\delta>0$, then
the BSDE (\ref{ch5-ll8}) admits a unique solution $\left(  Y,Z\right)  \in
S_{d}^{0}\left[  0,T\right]  \times\Lambda_{d\times k}^{0}\left(  0,T\right)
$ such that, for all $\delta>0$,%
\[
\mathbb{E}\sup\limits_{s\in\left[  0,T\right]  }e^{\delta p\nu_{s}}\left\vert
Y_{s}\right\vert ^{p}+\mathbb{E~}\left(
%TCIMACRO{\dint _{0}^{T}}%
%BeginExpansion
{\displaystyle\int_{0}^{T}}
%EndExpansion
e^{2\delta\nu_{s}}\left\vert Z_{s}\right\vert ^{2}ds\right)  ^{p/2}<\infty.
\]
Moreover, $\mathbb{P}-a.s.~\omega\in\Omega$,
\[
\left\vert Y_{t}\right\vert \leq C^{\prime}\left(  1+\left(  \mathbb{E}%
^{\mathcal{F}_{t}}~\left\vert \eta\right\vert ^{p}\right)  ^{1/p}\right)
,\quad~\text{for all }t\in\left[  0,T\right]  .
\]

\end{theorem}

\begin{proof}
Consider the projector operator $\pi:\Omega\times\left[  0,T\right]
\times\mathbb{R}^{d}\rightarrow\mathbb{R}^{d},$%
\[
\pi_{t}\left(  \omega,y\right)  =\pi\left(  \omega,t,y\right)  =y\left[
1-\left(  1-\frac{\beta_{t}\left(  \omega\right)  }{\left\vert y\right\vert
}\right)  ^{+}\right]  =\left\{
\begin{array}
[c]{ll}%
y, & \quad\text{if }\left\vert y\right\vert \leq\beta_{t}\left(
\omega\right)  ,\medskip\\
\dfrac{y}{\left\vert y\right\vert }\beta_{t}\left(  \omega\right)  , &
\quad\text{if }\left\vert y\right\vert >\beta_{t}\left(  \omega\right)  .
\end{array}
\right.
\]
Remark that, for all $y,y^{\prime}\in\mathbb{R}^{d}$, $\pi\left(  \cdot
,\cdot,y\right)  $ is a progressively measurable stochastic process,
$\left\vert \pi_{t}\left(  y\right)  \right\vert \leq\beta_{t}$ and%
\[
|\pi_{t}\left(  y\right)  -\pi_{t}(y^{\prime})|\leq|y-y^{\prime}|.
\]
Let $\tilde{\Phi}\left(  s,y,z\right)  =\Phi\left(  s,\pi_{s}\left(  y\right)
,z\right)  .$ The function is globally Lipschitz with respect to $\left(
y,z\right)  :$%
\begin{align*}
|\tilde{\Phi}\left(  s,y,z\right)  -\tilde{\Phi}(s,y^{\prime},z)|  &
=|\Phi\left(  s,\pi_{s}\left(  y\right)  ,z\right)  -\Phi(s,\pi_{s}(y^{\prime
}),z)|\\
&  \leq L(s)\left(  1+\left\vert \pi_{s}\left(  y\right)  \right\vert \vee
|\pi_{s}(y^{\prime})|\right)  |\pi_{s}\left(  y\right)  -\pi_{s}(y^{\prime
})|\\
&  \leq L(s)\left(  1+\beta_{s}\right)  |y-y^{\prime}|
\end{align*}
and%
\[
|\tilde{\Phi}\left(  s,y,z\right)  -\tilde{\Phi}(s,y^{\prime},z)|=\left\vert
\Phi\left(  s,\pi_{s}\left(  y\right)  ,z\right)  -\Phi(s,\pi_{s}\left(
y\right)  ,z^{\prime})\right\vert \leq\alpha_{s}\ell(s)|z-z^{\prime}|.
\]
Then, according to Theorem \ref{ch5-t2aa}, the BSDE%
\begin{equation}
Y_{t}=\eta+\int_{t}^{T}\tilde{\Phi}\left(  s,Y_{s},Z_{s}\right)  dQ_{s}%
-\int_{t}^{T}Z_{s}dB_{s},~,\;t\in\left[  0,T\right]  . \label{eq-localiz}%
\end{equation}
admits a unique solution $\left(  Y,Z\right)  \in S_{d}^{0}\left[  0,T\right]
\times\Lambda_{d\times k}^{0}\left(  0,T\right)  $ satisfying%
\[
\mathbb{E}\sup\limits_{s\in\left[  0,T\right]  }e^{\delta pV_{s}}\left\vert
Y_{s}\right\vert ^{p}+\mathbb{E}\left(
%TCIMACRO{\dint _{0}^{T}}%
%BeginExpansion
{\displaystyle\int_{0}^{T}}
%EndExpansion
e^{2\delta V_{s}}\left\vert Z_{s}\right\vert ^{2}ds\right)  ^{p/2}<\infty,
\]
where%
\[
V_{t}=%
%TCIMACRO{\dint _{0}^{t}}%
%BeginExpansion
{\displaystyle\int_{0}^{t}}
%EndExpansion
\left[  \kappa\left(  s\right)  +L\left(  s\right)  \left(  1+\beta
_{s}\right)  +\frac{1}{n_{p}}\ell^{2}\left(  s\right)  \right]  ds\leq C+C%
%TCIMACRO{\dint _{0}^{t}}%
%BeginExpansion
{\displaystyle\int_{0}^{t}}
%EndExpansion
L\left(  s\right)  \left[  \mathbb{E}^{\mathcal{F}_{s}}\left\vert
\eta\right\vert ^{p}\right]  ^{1/p}.
\]
Since we have%
\begin{align*}
\left\langle Y_{t},\tilde{\Phi}\left(  t,Y_{t},Z_{t}\right)  dQ_{t}%
\right\rangle  &  =\left\langle Y_{t},\Phi\left(  t,\pi_{t}\left(
Y_{t}\right)  ,Z_{t}\right)  dQ_{t}\right\rangle \\
&  \leq\left\vert Y_{t}\right\vert \rho(t)dQ_{t}+\left\vert Y_{t}\right\vert
^{2}\gamma(t)dQ_{t}+\frac{n_{p}}{4}\left\vert Z_{t}\right\vert ^{2}dt
\end{align*}
then $\left\vert Y_{t}\right\vert \leq\beta_{t}$ and, consequently,
$\tilde{\Phi}\left(  t,Y_{t},Z_{t}\right)  =\Phi\left(  t,Y_{t},Z_{t}\right)
$, that is $\left(  Y,Z\right)  $ is the unique solution of BSDE
(\ref{ch5-ll8}).

\hfill
\end{proof}

\subsection{Moreau-Yosida regularization of a convex function}

By $\nabla\varphi_{\varepsilon}$ we denote the gradient of the Yosida's
regularization $\varphi_{\varepsilon}$ of the function $\varphi$. More
precisely (see Br\'{e}zis \cite{Brezis:73}),%
\[
\varphi_{\varepsilon}(x)=\inf\,\{\frac{1}{2\varepsilon}|z-x|^{2}%
+\varphi(z):\;z\in\mathbb{R}^{d}\}=\dfrac{1}{2\varepsilon}|x-J_{\varepsilon
}x|^{2}+\varphi(J_{\varepsilon}x),
\]
where $J_{\varepsilon}x=x-\varepsilon\nabla\varphi_{\varepsilon}(x).$ The
function $\varphi_{\varepsilon}:\mathbb{R}^{d}\rightarrow\mathbb{R}$ is a
convex and differentiable one and it has the following main properties. For
all $x,y\in\mathbb{R}^{d},$ $\varepsilon>0:$%
\begin{equation}%
\begin{array}
[c]{ll}%
a)\quad & \nabla\varphi_{\varepsilon}(x)=\partial\varphi_{\varepsilon}\left(
x\right)  \in\partial\varphi(J_{\varepsilon}x),\text{ and }\varphi
(J_{\varepsilon}x)\leq\varphi_{\varepsilon}(x)\leq\varphi(x),\medskip\\
b)\quad & \left\vert \nabla\varphi_{\varepsilon}(x)-\nabla\varphi
_{\varepsilon}(y)\right\vert \leq\dfrac{1}{\varepsilon}\left\vert
x-y\right\vert ,\medskip\\
c)\quad & \left\langle \nabla\varphi_{\varepsilon}(x)-\nabla\varphi
_{\varepsilon}(y),x-y\right\rangle \geq0,\medskip\\
d)\quad & \left\langle \nabla\varphi_{\varepsilon}(x)-\nabla\varphi_{\delta
}(y),x-y\right\rangle \geq-(\varepsilon+\delta)\left\langle \nabla
\varphi_{\varepsilon}(x),\nabla\varphi_{\delta}(y)\right\rangle .
\end{array}
\label{sub6a}%
\end{equation}
If $0=\varphi\left(  0\right)  \leq\varphi\left(  x\right)  $ for all
$x\in\mathbb{R}^{d}$ then%
\begin{equation}%
\begin{array}
[c]{l}%
\left(  a\right)  \quad\quad0=\varphi_{\varepsilon}(0)\leq\varphi
_{\varepsilon}(x)\quad\text{and}\quad J_{\varepsilon}\left(  0\right)
=\nabla\varphi_{\varepsilon}\left(  0\right)  =0,\smallskip\\
\left(  b\right)  \quad\quad\dfrac{\varepsilon}{2}|\nabla\varphi_{\varepsilon
}(x)|^{2}\leq\varphi_{\varepsilon}(x)\leq\left\langle \nabla\varphi
_{\varepsilon}(x),x\right\rangle ,\quad\forall x\in\mathbb{R}^{d}.
\end{array}
\label{sub6c}%
\end{equation}

\begin{proposition}
\label{p12annexB}Let\ $\varphi:\mathbb{R}^{d}\rightarrow]-\infty,+\infty]$ be
a proper convex lower semicontinuous function such that $int\left(  Dom\left(
\varphi\right)  \right)  \neq\emptyset.$ Let $\left(  u_{0},\hat{u}%
_{0}\right)  \in\partial\varphi,$ $r_{0}\geq0$ and%
\[
\varphi_{u_{0},r_{0}}^{\#}\overset{def}{=}\sup\left\{  \varphi\left(
u_{0}+r_{0}v\right)  :\left\vert v\right\vert \leq1\right\}  .
\]
Then, for all $\,0\leq s\leq t$ and $dk\left(  t\right)  \in\partial
\varphi\left(  x\left(  t\right)  \right)  \left(  dt\right)  $,%
\begin{equation}
r_{0}\left(  \left\updownarrow k\right\updownarrow _{t}-\left\updownarrow
k\right\updownarrow _{s}\right)  +%
%TCIMACRO{\dint _{s}^{t}}%
%BeginExpansion
{\displaystyle\int_{s}^{t}}
%EndExpansion
\varphi(x(r))dr\leq%
%TCIMACRO{\dint _{s}^{t}}%
%BeginExpansion
{\displaystyle\int_{s}^{t}}
%EndExpansion
\left\langle x\left(  r\right)  -u_{0},dk\left(  r\right)  \right\rangle
+\left(  t-s\right)  \varphi_{u_{0},r_{0}}^{\#} \label{Ba6a}%
\end{equation}
and, moreover,%
\begin{equation}%
\begin{array}
[c]{l}%
r_{0}\left(  \left\updownarrow k\right\updownarrow _{t}-\left\updownarrow
k\right\updownarrow _{s}\right)  +%
%TCIMACRO{\dint _{s}^{t}}%
%BeginExpansion
{\displaystyle\int_{s}^{t}}
%EndExpansion
\left\vert \varphi(x(r))-\varphi\left(  u_{0}\right)  \right\vert dr\leq%
%TCIMACRO{\dint _{s}^{t}}%
%BeginExpansion
{\displaystyle\int_{s}^{t}}
%EndExpansion
\left\langle x\left(  r\right)  -u_{0},dk\left(  r\right)  \right\rangle
\smallskip\smallskip\\
\quad\quad\quad\quad\quad\quad\quad\quad\quad\quad\quad\quad\quad\quad+%
%TCIMACRO{\dint \nolimits_{s}^{t}}%
%BeginExpansion
{\displaystyle\int\nolimits_{s}^{t}}
%EndExpansion
(2\left\vert \hat{u}_{0}\right\vert \left\vert x(r)-u_{0}\right\vert
+\varphi_{u_{0},r_{0}}^{\#}-\varphi\left(  u_{0}\right)  )dr.
\end{array}
\label{Ba6b}%
\end{equation}

\end{proposition}

\subsection{Basic inequalities}

We shall derive some important estimations on the stochastic processes
$\left(  Y,Z\right)  \in S_{d}^{0}\left[  0,T\right]  \times\Lambda_{d\times
k}^{0}\left(  0,T\right)  $ satisfying for all $t\in\left[  0,T\right]  $,
$\mathbb{P}-a.s.~\omega\in\Omega$,%
\[
Y_{t}=Y_{T}+\int_{t}^{T}dK_{s}-\int_{t}^{T}Z_{s}dB_{s},
\]
with $K\in S_{d}^{0}$ be such that $K_{\cdot}\left(  \omega\right)  \in
BV_{loc}\left(  \left[  0,\infty\right[  ;\mathbb{R}^{d}\right)
,\;\mathbb{P}-a.s.~\omega\in\Omega.$ For more details concerning the results
found in this subsection one can consult Section 6.3.4 from Pardoux and
R\u{a}\c{s}canu \cite{Pardoux/Rascanu:09}.\medskip

\noindent\textbf{Backward It\^{o}'s formula. }\textit{If }$\psi\in
C^{1,2}\left(  \left[  0,T\right]  \times\mathbb{R}^{d}\right)  $\textit{,
then }$\mathbb{P}-a.s.~\omega\in\Omega$\textit{, for all }$t\in\left[
0,T\right]  $,%
\begin{equation}%
\begin{array}
[c]{l}%
\psi\left(  t,Y_{t}\right)  +%
%TCIMACRO{\dint _{t}^{T}}%
%BeginExpansion
{\displaystyle\int_{t}^{T}}
%EndExpansion
\left\{  \dfrac{\partial\psi}{\partial t}\left(  s,Y_{s}\right)  +\dfrac{1}%
{2}\mathbf{Tr}\left[  Z_{s}Z_{s}^{\ast}\psi_{xx}^{\prime\prime}\left(
s,Y_{s}\right)  \right]  \right\}  ds\medskip\\
\;\;\;\;\;=\psi\left(  T,Y_{T}\right)  +%
%TCIMACRO{\dint _{t}^{T}}%
%BeginExpansion
{\displaystyle\int_{t}^{T}}
%EndExpansion
\left\langle \psi_{x}^{\prime}\left(  s,Y_{s}\right)  ,dK_{s}\right\rangle -%
%TCIMACRO{\dint _{t}^{T}}%
%BeginExpansion
{\displaystyle\int_{t}^{T}}
%EndExpansion
\left\langle \psi_{x}^{\prime}\left(  s,Y_{s}\right)  ,Z_{s}dB_{s}%
\right\rangle
\end{array}
\label{ch5-if}%
\end{equation}

\noindent According to Lemma 2.35 from \cite{Pardoux/Rascanu:09}, if
$\psi\,:\left[  0,T\right]  \times\mathbb{R}^{d}\rightarrow\mathbb{R}$\textit{
}is a\textit{ }$C^{1}$-class function, convex in the second argument, then,
$\mathbb{P}-a.s.~\omega\in\Omega$, for every $t\in\left[  0,T\right]  $, the
following stochastic subdifferential inequality takes place:%
\begin{equation}
\psi(t,Y_{t})+%
%TCIMACRO{\dint _{t}^{T}}%
%BeginExpansion
{\displaystyle\int_{t}^{T}}
%EndExpansion
\dfrac{\partial\psi}{\partial t}\left(  s,Y_{s}\right)  ds\leq\psi(T,Y_{T})+%
%TCIMACRO{\dint _{t}^{T}}%
%BeginExpansion
{\displaystyle\int_{t}^{T}}
%EndExpansion
\left\langle \nabla\psi(s,Y_{s}),\,dK_{s}\right\rangle -%
%TCIMACRO{\dint _{t}^{T}}%
%BeginExpansion
{\displaystyle\int_{t}^{T}}
%EndExpansion
\left\langle \nabla\psi(s,Y_{s}),\,Z_{s}dB_{s}\right\rangle .
\label{subdiff ineq 1}%
\end{equation}
\medskip

\textbf{A fundamental inequality}\medskip\newline Let $\left(  Y,Z\right)  \in
S_{d}^{0}\left[  0,T\right]  \times\Lambda_{d\times k}^{0}\left(  0,T\right)
$ satisfying an identity of the form%
\begin{equation}
Y_{t}=Y_{T}+\int_{t}^{T}dK_{s}-\int_{t}^{T}Z_{s}dB_{s},\quad\;t\in\left[
0,T\right]  ,\quad\mathbb{P}-a.s.~\omega\in\Omega, \label{bsde-Fineq}%
\end{equation}
where $K\in S_{d}^{0}\left(  \left[  0,T\right]  \right)  $ and $K_{\cdot
}\left(  \omega\right)  \in BV\left(  \left[  0,T\right]  ;\mathbb{R}%
^{d}\right)  ,\;\mathbb{P}-a.s.~\omega\in\Omega.$\bigskip

\noindent Assume there exist

\begin{itemize}
\item $\quad D,R,N$ - three progressively measurable increasing continuous
stochastic processes with $D_{0}=R_{0}=N_{0}=0,$

\item $\quad V$ - a progressively measurable bounded variation continuous
stochastic process with $V_{0}=0,$

\item $\quad0\leq\lambda<1<p,$
\end{itemize}

\noindent such that, as measures on $\left[  0,T\right]  $, $\mathbb{P}%
-a.s.~\omega\in\Omega$,%
\begin{equation}
dD_{t}+\left\langle Y_{t},dK_{t}\right\rangle \leq\left[  \mathbf{1}_{p\geq
2}dR_{t}+|Y_{t}|dN_{t}+|Y_{t}|^{2}dV_{t}\right]  +\dfrac{n_{p}}{2}%
\lambda\left\vert Z_{t}\right\vert ^{2}dt, \label{Ch5-ip1}%
\end{equation}
where%
\[
n_{p}\overset{def}{=}1\wedge\left(  p-1\right)  \text{.}%
\]
Proposition 6.80 from Pardoux and R\u{a}\c{s}canu \cite{Pardoux/Rascanu:09}
yields the following important result.

\begin{proposition}
\label{ineq cond exp}If (\ref{bsde-Fineq}) and (\ref{Ch5-ip1}) hold, and
moreover%
\[
\mathbb{E}\left\Vert Ye^{V}\right\Vert _{T}^{p}<\infty,
\]
then
\index{inequality!backward stochastic}
there exists a positive constant $C_{p,\lambda},$ depending only upon $\left(
p,\lambda\right)  ,$ such that, $\mathbb{P}-a.s.~\omega\in\Omega$, for all
$t\in\left[  0,T\right]  $,%
\begin{equation}%
\begin{array}
[c]{r}%
\mathbb{E}^{\mathcal{F}_{t}}\sup\limits_{s\in\left[  t,T\right]  }\left\vert
e^{V_{s}}Y_{s}\right\vert ^{p}+\mathbb{E}^{\mathcal{F}_{t}}\left(
%TCIMACRO{\dint _{t}^{T}}%
%BeginExpansion
{\displaystyle\int_{t}^{T}}
%EndExpansion
e^{2V_{s}}dD_{s}\right)  ^{p/2}+\mathbb{E}^{\mathcal{F}_{t}}\left(
%TCIMACRO{\dint _{t}^{T}}%
%BeginExpansion
{\displaystyle\int_{t}^{T}}
%EndExpansion
e^{2V_{s}}\left\vert Z_{s}\right\vert ^{2}ds\right)  ^{p/2}\medskip\\
+\mathbb{E}^{\mathcal{F}_{t}}%
%TCIMACRO{\dint _{t}^{T}}%
%BeginExpansion
{\displaystyle\int_{t}^{T}}
%EndExpansion
e^{pV_{s}}\left\vert Y_{s}\right\vert ^{p-2}\mathbf{1}_{Y_{s}\neq0}%
dD_{s}+\mathbb{E}^{\mathcal{F}_{t}}%
%TCIMACRO{\dint _{t}^{T}}%
%BeginExpansion
{\displaystyle\int_{t}^{T}}
%EndExpansion
e^{pV_{s}}\left\vert Y_{s}\right\vert ^{p-2}\mathbf{1}_{Y_{s}\neq0}\left\vert
Z_{s}\right\vert ^{2}ds\medskip\\
\leq C_{p,\lambda}~\mathbb{E}^{\mathcal{F}_{t}}\mathbb{~}\left[  \left\vert
e^{V_{T}}Y_{T}\right\vert ^{p}+\left(
%TCIMACRO{\dint _{t}^{T}}%
%BeginExpansion
{\displaystyle\int_{t}^{T}}
%EndExpansion
e^{2V_{s}}\mathbf{1}_{p\geq2}dR_{s}\right)  ^{p/2}+\left(
%TCIMACRO{\dint _{t}^{T}}%
%BeginExpansion
{\displaystyle\int_{t}^{T}}
%EndExpansion
e^{V_{s}}dN_{s}\right)  ^{p}\right]  .
\end{array}
\label{ch5-b2}%
\end{equation}
In addition, if $R=N=0,$ then, for all $t\in\left[  0,T\right]  $,%
\begin{equation}
e^{pV_{t}}\left\vert Y_{t}\right\vert ^{p}\leq\mathbb{E}^{\mathcal{F}_{t}%
}e^{pV_{T}}\left\vert Y_{T}\right\vert ^{p},\quad\mathbb{P}-a.s.~\omega
\in\Omega. \label{ch5-b2a}%
\end{equation}

\end{proposition}

\begin{corollary}
\label{Ch5-p1-cor1} Under the assumptions of Proposition \ref{ineq cond exp},
if $V$ is a determinist process and $\sup_{s\geq0}\left\vert V_{s}\right\vert
\leq c$ then, $\mathbb{P}-a.s.~\omega\in\Omega$, for all $t\in\left[
0,T\right]  $,%
\[%
\begin{array}
[c]{l}%
\mathbb{E}^{\mathcal{F}_{t}}\sup\limits_{s\in\left[  t,T\right]  }\left\vert
Y_{s}\right\vert ^{p}+\mathbb{E}^{\mathcal{F}_{t}}\left(
%TCIMACRO{\dint _{t}^{T}}%
%BeginExpansion
{\displaystyle\int_{t}^{T}}
%EndExpansion
\left\vert Z_{s}\right\vert ^{2}ds\right)  ^{p/2}\medskip\\
\quad\quad\quad\quad\quad\quad\leq C_{p,\lambda}e^{2c}\mathbb{E}%
^{\mathcal{F}_{t}}\left[  \left\vert Y_{T}\right\vert ^{p}+\left(
%TCIMACRO{\dint _{t}^{T}}%
%BeginExpansion
{\displaystyle\int_{t}^{T}}
%EndExpansion
\mathbf{1}_{p\geq2}dR_{s}\right)  ^{p/2}+\left(
%TCIMACRO{\dint _{t}^{T}}%
%BeginExpansion
{\displaystyle\int_{t}^{T}}
%EndExpansion
dN_{s}\right)  ^{p}\right]  .
\end{array}
\]

\end{corollary}

\begin{proposition}
[See Proposition 6.69 from \cite{Pardoux/Rascanu:09}]\label{exp prop ineq}%
\textit{Let }$\delta\in\left\{  -1,1\right\}  $ and consider\textit{
}$Y,K,A:\Omega\times\mathbb{R}_{+}\rightarrow\mathbb{R}\;$and $G:\Omega
\times\mathbb{R}_{+}\rightarrow\mathbb{R}^{k}$ four progressively measurable
stochastic processes such that%
\[%
\begin{array}
[c]{rll}%
i)\;\; &  & Y,K,A\;\text{are continuous stochastic processes,}\\
ii)\;\; &  & A_{\cdot},K_{\cdot}\in BV_{loc}\left(  \left[  0,\infty\right[
;\mathbb{R}\right)  ,\;A_{0}=K_{0}=0,\;\mathbb{P}-a.s.~\omega\in\Omega
\text{,}\\
iii)\;\; &  &
%TCIMACRO{\dint _{t}^{s}}%
%BeginExpansion
{\displaystyle\int_{t}^{s}}
%EndExpansion
\left\vert G_{r}\right\vert ^{2}dr<\infty,\;\mathbb{P}-a.s.~\omega\in
\Omega,\;\forall0\leq t\leq s.
\end{array}
\]
\textit{If, for all }$0\leq t\leq s$,%
\[
\delta\left(  Y_{t}-Y_{s}\right)  \leq\int_{t}^{s}\left(  dK_{r}+Y_{r}%
dA_{r}\right)  +\int_{t}^{s}\left\langle G_{r},dB_{r}\right\rangle
,\quad\mathbb{P}-a.s.~\omega\in\Omega,
\]
\textit{then}%
\[
\delta\left(  Y_{t}e^{\delta A_{t}}-Y_{s}e^{\delta A_{s}}\right)  \leq\int
_{t}^{s}e^{\delta A_{r}}dK_{r}+\int_{t}^{s}e^{\delta A_{r}}\left\langle
G_{r},dB_{r}\right\rangle ,\quad\mathbb{P}-a.s.~\omega\in\Omega.
\]

\end{proposition}


\bigskip

\begin{thebibliography}{99}                                                                                               %


\bibitem {Billingsley:99}Billingsley, P.\textbf{ -} \textit{Convergence of
Probability Measures}, New York, NY: John Wiley \& Sons, 1999.$\vspace
{-0.08in}$

\bibitem {Bismut:73}Bismut, J.M.\textbf{ -} \textit{Conjugate Convex Functions
in Optimal Stochastic Control}, J. Math. Anal. Appl.,44, pp. 384-404,
1973.$\vspace{-0.08in}$

\bibitem {Boufoussi/Casteren:04}Boufoussi, B.; van Casteren, J. - \textit{An
approximation result for a nonlinear Neumann boundary value problem via
BSDEs}, Stochastic Processes and their Applications, Volume 114, pp. 331-350,
2004.$\vspace{-0.08in}$

\bibitem {Brezis:73}Br\'{e}zis, H. \textbf{-} \textit{Op\'{e}rateurs Maximaux
Monotones et Semigroupes de Contractions Dans les Espaces de Hilbert},
North-Holland, Amsterdam, 1973.$\vspace{-0.08in}$

\bibitem {Buckdahn/Engelbert/Rascanu:04}Buckdahn, R.; Engelbert, H.-J.;
R\u{a}\c{s}canu, A.\textbf{ -} \textit{On weak solutions of backward
stochastic differential equations}, Teor. Veroyatn. Primen. 49, 1, pp. 70-108,
2004.$\vspace{-0.08in}$

\bibitem {Dupuis/Ishii:93}Dupuis, P.; Ishii, H.\textbf{ -} \textit{SDEs with
oblique reflection on nonsmooth domains},\emph{ }Ann. Probab. 21, no. 1, pp.
554-580, 1993.$\vspace{-0.08in}$

\bibitem {Friedman:75}Friedman, A.\textbf{ -} \textit{Stochastic Differential
Equations and Applications I}, Academic Press, 1975.$\vspace{-0.08in}$

\bibitem {Gassous/Rascanu/Rotenstein:12}Gassous, A; R\u{a}\c{s}canu. A;
Rotenstein, E. \textbf{-} \textit{Stochastic variational inequalities with
oblique subgradients}, Stochastic Processes and their Applications, Volume
122, Issue 7, pp. 2668-2700, 2012.$\vspace{-0.08in}$

\bibitem {Hu/Tang:10}Hu, Y.; Tang, S. - \textit{Multi-dimensional BSDE with
oblique reflection and optimal switching}, Probab. Theory and Related Fields,
Volume 147, Issue 1-2, pp. 89-121, 2010.$\vspace{-0.08in}$

\bibitem {Ikeda/Watanabe:81}Ikeda, N.;\ Watanabe, S. \textbf{-}
\textit{Stochastic differential equations and diffusion processes},
North--Holland/Kodansha, 1981.$\vspace{-0.08in}$

\bibitem {Jakubowski:97}Jakubowski, A. \textbf{-} \textit{A non-Skorokhod
topology on the Skorokhod space}, Electron. J. Probab. 2 (4), pp. 1-21,
1997.$\vspace{-0.08in}$

\bibitem {LeJai:02}LeJay, A. \textbf{-} \textit{BSDE driven by Dirichlet
process and semi-linear parabolic PDE. Application to homogenization},
Stochastic Processes and Their Applications (1), pp. 1-39, 2002.$\vspace
{-0.08in}$

\bibitem {Lions/Sznitman-84}Lions, P.-L.; Sznitman, A.\textbf{ -}
\textit{Stochastic differential equations with reflecting boundary
conditions}, Comm. Pure Appl. Math. 37, no. 4, pp. 511-537, 1984.$\vspace
{-0.08in}$

\bibitem {Maticiuc/Pardoux/Rascanu/Zalinescu:10}Maticiuc, L.; Pardoux, E.;
R\u{a}\c{s}canu, A.; Z\u{a}linescu, A. - \textit{Viscosity solutions for
systems of parabolic variational inequalities}, Bernoulli Volume 16, Number 1,
pp. 258-273, 2010.$\vspace{-0.08in}$

\bibitem {Maticiuc/Rascanu:07}Maticiuc, L., R\u{a}\c{s}canu, A. \textbf{-}
\textit{Backward stochastic generalized variational inequality}, Applied
analysis and differential equations, pp. 217-226, World Sci. Publ.,
Hackensack, NJ, 2007.$\vspace{-0.08in}$

\bibitem {Maticiuc/Rascanu:10}Maticiuc, L.; R\u{a}\c{s}canu, A. - \textit{A
stochastic approach to a multivalued Dirichlet-Neumann problem},Stochastic
Processes and their Applications 120, pp. 777-800, 2010.$\vspace{-0.08in}$

\bibitem {Pardoux:99}Pardoux, E. \textbf{-}\textit{ BSDEs, weak convergence
and homogenization of semilinear PDEs}, Nonlinear Analysis, Differential
Equations and Control, NATO Science Series Volume 528, pp 503-549,
1999.$\vspace{-0.08in}$

\bibitem {Pardoux/Peng:90}Pardoux, E., Peng, S. \textbf{-} \textit{Solution of
a Backward Stochastic Differential Equation}, Systems and Control Letters, 14,
pp. 55-61, 1990.$\vspace{-0.08in}$

\bibitem {Pardoux/Peng:92}Pardoux, E., Peng, S. \textbf{-} \textit{Backward
SDE's and quasilinear parabolic PDE's}, Stochastic PDE and Their Applications,
LNCIS 176, Springer, pp. 200-217, 1992.$\vspace{-0.08in}$

\bibitem {Pardoux/Rascanu:98}Pardoux, E.; R\u{a}\c{s}canu, A. \textbf{-}%
\textit{\ Backward stochastic differential equations with subdifferential
operator and related variational inequalities}, Stochastic Processes and Their
Applications, vol. 76, no. 2, pp. 191-215, 1998.$\vspace{-0.08in}$

\bibitem {Pardoux/Rascanu:99}Pardoux, E.; R\u{a}\c{s}canu, A. \textbf{-}%
\textit{\ Backward stochastic variational inequalities}, Stochastics
Stochastics Rep., 67(3-4), pp. 159-167, 1999.$\vspace{-0.08in}$

\bibitem {Pardoux/Rascanu:09}Pardoux, E.; R\u{a}\c{s}canu, A. \textbf{-}
\textit{SDEs, BSDEs and PDEs,}\emph{ }book, accepted for publication in
Springer, 2013.$\vspace{-0.08in}$

\bibitem {Rascanu:96}R\u{a}\c{s}canu, A. \textbf{-} \textit{Deterministic and
stochastic differential equations in Hilbert spaces involving multivalued
maximal monotone operators}, Panamer. Math. J. 6, no. 3, pp. 83-119,
1996.$\vspace{-0.08in}$

\bibitem {Rascanu/Rotenstein:11}R\u{a}\c{s}canu, A.; Rotenstein, E. \textbf{-}
\textit{The Fitzpatrick function-a bridge between convex analysis and
multivalued stochastic differential equations}, J. Convex Anal. 18 (1), pp.
105-138, 2011.$\vspace{-0.08in}$

\bibitem {Tanaka:79}Tanaka, H. \textbf{-} \textit{Stochastic differential
equations with reflecting boundary condition in convex regions}, Hiroshima
Math. J., pp. 163-177, 1979.$\vspace{-0.08in}$
\end{thebibliography}
\end{document}